\documentclass[11pt]{amsart}
\usepackage{amssymb,latexsym}
\usepackage{amsmath}
\newtheorem{theorem}{Theorem}[section] 
\newtheorem{lemma}[theorem]{Lemma}     

\newtheorem{prop}[theorem]{Proposition}

\newtheorem{cj}{Conjecture}  

\newtheorem{remark}[theorem]{Remark}

\newcommand{\eeq}{\end{equation}}
\newcommand{\beql}[1]{\begin{equation}\label{#1}}
\newcommand{\eqn}[1]{(\ref{#1})}

\newcommand{\ABC}{{abc}}
\newcommand{\XYZ}{{xyz}}

\newcommand{\imod}{{\bmod}}

\newcommand{\RR}{{\mathbb R}}
\newcommand{\TT}{{\mathbb T}}
\newcommand{\ZZ}{{\mathbb Z}}

\newcommand{\sP}{{\mathcal P}}

\newcommand{\sS}{{\mathcal S}}

\newcommand{\fS}{{\mathfrak S}}
\newcommand{\fm}{{\mathfrak m}}
\newcommand{\fM}{{\mathfrak M}}

\title {Counting Smooth Solutions to the Equation $A+B=C$} 

\author{J. C. Lagarias}
\address{Department of Mathematics, University of Michigan,
Ann Arbor, MI 48109-1043,USA}
\email{lagarias@umich.edu}

\thanks{Research of the first author supported by NSF grants DMS-0500555 and DMS-0801029
and the second author by NSF grant  DMS-0500711and DMS-1001068.}
\author{ K. Soundararajan}
\address{Department of Mathematics, Stanford University,
Stanford, CA 94305 USA}
\email{ksound@stanford.edu}

\date{February 16, 2011}

\makeatletter
\def\imod#1{\allowbreak\mkern10mu({\operator@font mod}\,\,#1)}
\makeatother

\begin{document}

\begin{abstract}
This paper studies  integer
 solutions
 to the $\ABC$ equation $A+B=C$
in which none of $A, B, C$ have a large
prime factor. We set 
$H(A, B,C) = \max(|A|, |B|, |C|)$,
and consider primitive solutions (${\rm gcd}(A, B, C)=1$)
having  no prime factor $p$
larger than 
$(\log H(A, B,C))^{\kappa}$, for  a given finite  $\kappa$.
On the assumption that the Generalized
Riemann hypothesis (GRH) holds,  
we show that for any $\kappa > 8$ 
there are infinitely many  such primitive
solutions having no prime factor  larger than $(\log H(A, B, C))^{\kappa}$. 
We obtain in this range 
an asymptotic formula for the number of such 
suitably weighted primitive solutions.


\end{abstract}

\maketitle


%
%
%
%
\setlength{\baselineskip}{1.0\baselineskip}

\section{Introduction}

\noindent A recurring topic of investigation in number theory is the relation between additive 
and multiplicative structures of integers.  A celebrated example is the $\ABC$-conjecture of 
Masser \cite{Mas85} and Oesterl{\' e} \cite{Oe88}, cf.  \cite[Chap. 12]{BG06}.  
In its weak form, the $\ABC$-conjecture 
asserts that there is a constant $\kappa_1> 0$ such that for any $\epsilon >0$ there are 
only finitely many solutions to the equation $A+B=C$ with $ABC\neq 0$, 
g.c.d. $(A, B, C)=1$ and such that 
$$ 
\max(|A|, |B|, |C|) \le \Big(\prod_{p|ABC} p\Big)^{\kappa_1 -\epsilon}. 
$$ 
One may construct examples to show that $\kappa_1$, if it exists, cannot 
be smaller than $1$ (see Stewart and Tijdeman \cite{ST86}), and the strong form of the $\ABC$-conjecture postulates that in fact $\kappa_1=1$ is permissible.  
In this paper, we study a different statistic related to the prime factorization of $ABC$.  
In place of the {\sl radical} $\prod_{p|ABC} p$, we study the {\sl smoothness} 
$\max_{p|ABC} p$.  In \cite{LS-JA} we formulated the following conjecture, which we term the 
$\XYZ$ conjecture.  \smallskip

%
%
\paragraph{\bf \XYZ ~conjecture (weak form).} {\em There exists a
constant $\kappa_0> 0$ such that the following
hold.

 (a) For each  $\epsilon >0$ there are only finitely many solutions $(X,Y,Z)$ to  
the equation $X+Y=Z$ with g.c.d$(X,Y,Z)=1$ and 
\beql{108}
\max_{p|XYZ} p < (\log \max(|X|,|Y|,|Z|))^{\kappa_0 - \epsilon}.
\eeq

 (b)  For
each $\epsilon >0$ there are infinitely many solutions $(X,Y,Z)$   to 
the equation $X+Y =Z$ with g.c.d.$(X,Y,Z)=1$ and 
\beql{109}
\max_{p|XYZ} p  < (\log \max(|X|, |Y|, |Z|))^{\kappa_0 + \epsilon}.
\eeq
} 

We shall call a solution $A+B=C$ primitive if g.c.d.$(A,B,C)=1$.  The 
restriction to primitive solutions in the $\ABC$ and $\XYZ$ conjectures 
is needed to exclude examples like $a+a=2a$ where $a$ is a 
high perfect power, or $a$ is very smooth.

For any primitive solution $(X, Y, Z)$ to $X+Y=Z$
we define its {\em smoothness exponent} $\kappa_0(X, Y, Z)$ by
\beql{109d}
\kappa_0(X, Y, Z) := \frac{\log \max_{p|XYZ} p} {\log \log \max(|X|, |Y|, |Z|)}.
\eeq
Our interest then is in the $\XYZ$-smoothness exponent $\kappa_0$ which is defined as 
the $\lim \inf$ of $\kappa_0(X,Y,Z)$ as 
$\max(|X|, |Y|, |Z|) \to \infty$.  
{\sl A priori} we have $0 \le \kappa_0 \le + \infty$, and the weak form 
of the $\XYZ$-conjecture asserts that it is positive and finite.  We next give a 
heuristic for the weak $\XYZ$ conjecture which also suggests a plausible 
value for $\kappa_0$.  \smallskip

\paragraph{\bf \XYZ ~conjecture (strong form).} {\em The $\XYZ$-smoothness
exponent $\kappa_0 $ equals $3/2$.
}\\

A natural number $n$ is said to be $y$-smooth if all its prime factors lie 
below $y$.    Throughout we shall  let  $\sS(y)$ denote
the set of $y$-smooth numbers, and $\Psi(x,y)$ shall count the number of  positive integers
below $x$ lying in $\sS(y)$.     

Consider all the triples $(X, Y, -Z)$  drawn from the interval
$[1, H]$ but restricted to having all prime factors
smaller than $(\log H)^{\kappa}$. We wish to find solutions to
$X+Y-Z=0$.  There are
$\Psi(H, (\log H)^{\kappa})^3$ such triples, each
having a sum  $X+Y-Z$ that falls in the interval $[-H, 2H]$.
If these sums were randomly distributed, the 
chance that the value $0$ is hit 
might be expected to  be  approximately proportional to
\beql{109aa}
P(H, \kappa) := \frac{\Psi(H, (\log H)^{\kappa})^3}{H}.
\eeq
It is known that (see (\ref{S801}) below) for fixed $\kappa>1$, one has
\beql{109a}
\Psi(x, (\log x)^{\kappa}) = x^{1 - \frac{1}{\kappa} +o(1)},  
\eeq
as $x \to \infty$.
Thus for $\kappa >1$ the number of such triples $(X,Y,Z)$ 
is at most $\Psi(H,(\log H)^{\kappa})^3 = H^{3(1-\frac{1}{\kappa}+o(1))}$, 
and if $\kappa < \frac 32$ this is $< H^{1-\epsilon}$,
so that $P(H, \kappa) = H^{-\epsilon}$.  This leads us to believe that 
 $\kappa_0 \ge \frac 32$.  

We derive a  matching heuristic lower bound for the number of relatively prime triples.  
Take $X$ to be a number composed of exactly $K:=[\log H/(\kappa \log \log H)]$ 
distinct primes all below $(\log H)^{\kappa}$.  Using Stirling's formula 
there 
are ${\left({{\pi((\log H)^{\kappa})}\atop{K}} \right)}= H^{1-1/\kappa+o(1)}$ 
such values of $X$ all lying below $H$. Given $X$,  choose 
$Y$ to be a number composed of exactly $K$ distinct primes 
below $(\log H)^{\kappa}$,  but avoiding the primes dividing $X$.  
There are ${\left({{\pi((\log H)^{\kappa})-K}\atop{K}}\right)} = H^{1-1/\kappa+o(1)}$ 
such values of $Y$.  Finally choose $Z$ to be a number composed 
of exactly $K$ distinct primes below $(\log H)^{\kappa}$ avoiding 
the primes dividing $X$ and $Y$.  There are 
${\left({{\pi((\log H)^{\kappa})-2K}\atop{K}}\right)} = H^{1-1/\kappa+o(1)}$ 
such values of $Z$.  We conclude therefore that 
there are at least $H^{3-3/\kappa+o(1)}$ such  triples, 
and hence we expect that $\kappa_0 \le \frac 32$.

In Theorem 1.1 of  \cite{LS-JA} we observed  that lower bounds for the $\XYZ$ smoothness
exponent follow from the $\ABC$ conjecture.\smallskip

%
\begin{prop}\label{pr01}
The weak form of the $\ABC$-conjecture implies that
the $xyz$ smoothness exponent satisfies $\kappa_0 \ge \kappa_1$.
In particular,  the  strong form of the $\ABC$-conjecture implies that
$\kappa_0 \ge 1$.
\end{prop}

It is interesting to note that even the strong form of the $\ABC$ conjecture is insufficient to imply the 
conjectured  lower bound $\kappa_0 \ge \frac{3}{2}$  above on the $\XYZ$ exponent.

This paper studies  the upper bound part  of the $\XYZ$-conjecture.  
Assuming the truth of the  {\em Generalized Riemann
Hypothesis (GRH)}, which states that all non-trivial zeros of 
the Riemann zeta function and Dirichlet $L$-functions lie on the critical line ${\rm Re}(s)= \frac{1}{2}$, 
we shall show that $\kappa_0 \le 8$.   \smallskip

%
%
\begin{theorem}~\label{th12}
Assume the truth of the Generalized Riemann Hypothesis (GRH). Then for each 
$\epsilon > 0$ there are infinitely many primitive
solutions $(X,Y,Z)$ to $X+Y=Z$ such that  
all the primes dividing $XYZ$ are smaller than 
$(\log \max(|X|, |Y|, |Z|))^{8+\epsilon}$.  
In other words,  $\kappa_0 \le 8$.
\end{theorem}

This result is an immediate consequence of the following  stronger result, 
which gives a lower bound for the number of  primitive solutions in this range.\smallskip

%
%

\begin{theorem}~\label{th13}
{\em (Counting Primitive Smooth Solutions)}
Assume the truth of the Generalized Riemann Hypothesis (GRH). Then for each fixed
$\kappa > 8$ the number of primitive integer solutions $N^{\ast}(H, \kappa)$ 
 to $X+Y=Z$ with $0\le X, Y, Z \le H$ and such that the largest prime factor of $XYZ$ is 
$< (\log H)^{\kappa}$ satisfies
\beql{131}
N^{\ast}(H, \kappa) \ge  {\fS}_{\infty}\Big(1-\frac{1}{\kappa}\Big) 
{\fS}_{f}^*\Big(1-\frac{1}{\kappa}, (\log H)^{\kappa}\Big) 
\frac{\Psi(H, (\log H)^{\kappa})^3}{H}(1 + o(1)),
\eeq
as $H \to \infty$. Here the ``archimedean singular series" (more properly, 
``singular integral") ${\fS}_{\infty}(c)$ is defined, for $c>\frac 13$, by 
\beql{132}
{\fS}_{\infty}(c):= c^3 \int_{0}^{1} \int_{0}^{1-t_1} 
(t_1 t_2 (t_1+t_2))^{c-1} dt_1 dt_2,
\eeq
and the ``primitive non-archimedean singular series" $\fS_f^*(c,y)$ is defined by 
\beql{133}
{\fS}_{f}^*(c,y):= \prod_{p\le y} 
\Big( 1+ \frac{p-1}{p(p^{3c-1}-1)} \Big(\frac{p-p^c}{p-1}\Big)^3 \Big)\Big(1-\frac{1}{p^{3c-1}}\Big)
\prod_{p>y}\Big(1 - \frac{1}{(p-1)^2}\Big).
\eeq
\end{theorem}

We expect that the lower bound given by the right side of  \eqn{131} should give an  asymptotic
formula for $N^{*}(H, \kappa)$ in this range of $\kappa$, 
and that proving this should be accessible by elaboration of the methods of this paper.
The estimate (\ref{131}) is in accordance with the heuristic \eqn{109aa}
which would have predicted a main term of $\Psi(H,(\log H)^{\kappa})^3/H$.  
In the range $\kappa >8$ we see that the main term in (\ref{131}) differs from the 
heuristic only by the factor $\fS_\infty(1-1/\kappa) \fS_f^*(1-1/\kappa, (\log H)^{\kappa})$.
An argument below  shows this factor is bounded away from $0$ and $\infty$,
and for fixed $\kappa$ it approaches 
a constant (depending on $\kappa$) as $H \to \infty$.
As $\kappa\to \infty$, this constant factor approaches $\tfrac 12$, and the main term 
$\tfrac 12 \Psi(H, (\log H)^{\kappa})^3/H$ is the expected number of solutions to 
$X+Y=Z$ when $X$, $Y$ and $Z$ are drawn from a random subset of $[1,H]$ 
with cardinality $\Psi(H,(\log H)^\kappa)$.   Thus our  heuristic is very accurate in the range 
$\kappa \to \infty$.  

%

The ``main term" on the right side of \eqn{131} is
 well-defined   in  the range $\kappa> \frac{3}{2}$ where the heuristic above is expected to apply.
 Here $\kappa> \frac{3}{2}$ corresponds to $c> \frac{1}{3}$,
 and the  ``archimedean singular integral"  \eqn{132}
 defines  an analytic function on the half-plane Re$(c)> \frac{1}{3}$ which
 diverges at $c=\frac{1}{3}$, while 
 the   ``non-archimedean singular series"   $\fS_{f}^{*}(c, y)$ is well-defined for all $c>0$.
 The archimedean singular series  is uniformly
 bounded on any half-plane Re$(c)> \frac{1}{3} +\epsilon.$ 
 For the non-archimedean singular series, we find that its limiting
  behavior as $y= (\log H)^{\kappa} \to \infty$
  changes at the threshold value $\kappa=2$, 
  corresponding to $c = \frac{1}{2}$. Namely, one has
 \beql{135}
 \lim_{H \to \infty} {\fS}_{f}^{\ast}(1-\frac{1}{\kappa}, (\log H)^{\kappa}) = 
 \left\{ 
\begin{array}{ll} {\fS}_{f}^{\ast}(1-\frac{1}{\kappa}) &~~\mbox{for}~~ \kappa>2,\\
~&~\\
0 &~~\mbox{for} ~~ 0<\kappa \le 2,
\end{array}
\right.
\eeq
 where for $c> \frac{1}{2}$ we set
 \beql{321a}
 \fS_{f}^{\ast}(c) := \prod_{p} \Big( 1 + \frac{1}{p^{3c-1}}\Big( \frac{p-1}{p}\Big(\frac{p-p^c}{p-1}\Big)^3 -1\Big)\Big).
 \eeq
 (This follows from \eqn{133}).
 The Euler product \eqn{321a} converges absolutely
 and defines an analytic function ${\fS}_{f}^{\ast}(c)$
 on the
 half-plane Re$(c) > \frac{1}{2}$; 
 this function is uniformly bounded on any
 half-plane Re$(c)> \frac{1}{2}+\epsilon$, 
 Furthermore for values  corresponding the
 $2 \le \kappa < \infty$ (i.e. $\frac{1}{2} < c < 1$) the ``non-archimedean
 singular series"  $\fS_{f}(c, y)$ remains bounded away from $0$. We conclude that for $2 < \kappa < \infty$
 the ``main term" estimate for $N^{\ast}(H, \kappa)$
 agrees with the prediction of the heuristic argument given earlier.
 In the region $1 < \kappa\le 2$, 
  although \eqn{321a} gives $\fS_{f}^{\ast}(1-\frac{1}{\kappa}, (\log H)^{\kappa}) \to 0$
 as $H \to \infty$, nevertheless one can show
 \beql{323a}
  {\fS}_{f}^{\ast}(1-\frac{1}{\kappa}, (\log H)^{\kappa}) \gg \exp ( -(\log H)^{2-\kappa}). 
    \eeq
 This bound implies that $\fS_{f}^{\ast}(1-\frac{1}{\kappa}, (\log H)^{\kappa}) \gg H^{-\epsilon}$
for any $\epsilon>0$. A consequence is that  for $\frac{3}{2} < \kappa \le 2$ 
 the ``main term" on the right side of \eqn{132} is still of the same order $H^{2- \frac{3}{\kappa} +o(1)}$ as the
heuristic predicts.
  Thus  it could still  be the case that this  ``main term" 
  gives a correct order of magnitude estimate  for $N^{\ast}(H, \kappa)$ even in this range.

Next we compare the number $N^{\ast}(H, \kappa)$ of primitive smooth solutions
with the total number $N(H,\kappa)$ of smooth solutions below $H$. Now
$N(H, \kappa)$ already has a 
 contribution coming from smooth multiples of the solution $(X, Y, Z) =(1, 1, 2)$ that gives
 \beql{185}
 N(H, \kappa) \ge \Psi( \tfrac{1}{2} H, (\log H)^{\kappa}) \ge H^{1- \frac{1}{\kappa} +o(1)}, ~~\mbox{as}~~H \to \infty.
 \eeq
 For  $1 \le \kappa <2$ 
 this  lower bound exceeds the heuristic estimate $H^{2- \frac{3}{\alpha} +o(1)}$
 for  $N^{\ast}(H, \kappa)$   by a positive power of $H$. 
 It follows that the heuristic given for primitive smooth solutions
 should not apply  to  smooth solutions $N(H, \kappa)$ for $1 < \kappa < 2$,
 and furthermore it indicates that on this range the density of  
   primitive smooth solutions  in the set of
  all smooth solutions below $H$ will approach zero as $H \to \infty$.

We may  consider  for more general $\kappa$ the limiting behavior 
as $H \to \infty$ of the
relative density of primitive smooth solutions.
Here we conjecture there is a threshold value at $\kappa =3$
where this behavior changes qualitatively.\smallskip

%
%

\begin{cj}\label{cj14}
({\em Relative Density of Primitive Solutions})
There holds
\beql{370}
 \lim_{H \to \infty} \frac{ N^{\ast}(H, \kappa)}{N(H, \kappa)}= 
 \left\{ 
\begin{array}{ll} \frac{1}{\zeta(2- \frac{3}{\kappa})},& ~~~\mbox{for}~~ 3 < \kappa < \infty,\\
~&~\\
0 &~~\mbox{for} ~~ 1<\kappa \le 3.
\end{array}
\right.
\eeq
\end{cj}

 As evidence in favor of  this conjecture,
 Theorem \ref{th23a} below  shows, assuming GRH, that 
 a weighted version of this conjecture holds 
 for $\kappa>8$.
  Further evidence is the fact that  for each $\kappa > 3$ 
  the ratios of the conjectured ``main terms" in the 
 asymptotic formulas for these quantities have the limiting value $\zeta(2-\frac{3}{\kappa})$
 as $H \to \infty$, a
 result implied by \eqn{423} below. Finally, the discussion above
 gives support for its truth on the range $1 < \kappa \le 2$.

In \S2 we describe the main technical results from which the theorems above are derived. 
 Our main estimate (Theorem~\ref{th21}) gives an asymptotic formula with error term 
which counts weighted (primitive and imprimitive) integer solutions to the 
$\XYZ$-equation in the 
range $\kappa> 8$.
This result will be established using the  Hardy-Littlewood method (\cite{Va97})
combined with the Hildebrand-Tenenbaum saddle point method (\cite{HT86}, \cite{Hi86}, \cite{HT93}) for estimating the size of  $\Psi(x, y)$.  We then derive a
weighted count of primitive solutions
(Theorem~\ref{th121}) using 
inclusion-exclusion.  Theorem ~\ref{th13}  is deduced from Theorem \ref{th121}. 
  It would 
be interesting to see whether our main results could be made unconditional.  At the moment, 
the best known unconditional results are due to Balog and Sark{\" o}zy 
\cite{BS84a}, \cite{BS84b} who showed (in a closely related problem) for any large $N$,
 there are solutions to $X+Y+Z=N$ with the largest prime factor of $XYZ$ being smaller 
 than $\exp(3\sqrt{\log N \log \log N})$.

Our problem may also be viewed as a special case of the $S$-unit equation. 
Given a finite set of primes $S$, one can consider relatively prime
solutions to the $S$-unit equation $X+Y=Z$ 
where all prime factors of $XYZ$ are in the 
set $S$.  In 1988 Erd\H{o}s, Stewart and Tijdeman \cite{EST88} showed the existence
of collections of primes $S$ with $|S|=s$ such
that the $S$-unit equation $X+Y= Z$ has ``exponentially many''
solutions, namely at least 
$\exp((4- \epsilon) s^{\frac{1}{2}} (\log s)^{-\frac{1}{2}} )$
solutions, for $s \ge s_0(\epsilon)$ sufficiently large.
Recently Konyagin and the second author \cite{KS07}
improved this construction, to show that there exist $S$ such
that the 
$S$-unit equation has at least
$\exp(s^{2- \sqrt{2}- \epsilon})$ solutions.  In the other 
direction, Evertse \cite[Theorem 1]{Ev84} has shown that the number of solutions to the 
$S$-unit equation is at most $3 \times 7^{2s+3}$.

In the constructions above  the sets of primes  $S$ were tailored to have 
large numbers of solutions. However the  simplest set of such primes to 
consider is the initial segment of primes 
$S =\sP(y) :=\{ p: p ~\mbox{prime}~, p \le y \}.$
Erd\H{o}s, Stewart and Tijdeman conjectured (\cite[p. 49, top]{EST88}) 
that a similar property should hold in this case,
asserting  that for $s= |S|$ and each $\epsilon>0$
there should be at least $\exp( s^{\frac{2}{3} - \epsilon})$
$S$-unit solutions to $X+Y=Z$ and at most $\exp( s^{\frac{2}{3} + \epsilon})$
such solutions, for all $s > s_0(\epsilon).$  Their conjecture was motivated by a 
heuristic similar to the one given above for the strong $\XYZ$-conjecture.

As an easy
consequence of
Theorem~\ref{th13} we deduce, conditional on $GRH$,  a weak form of this conjecture,
at the end of \S2. \smallskip
%
%
\begin{theorem}~\label{th15}
Assume the truth of the Generalized Riemann Hypothesis (GRH).
Let $S$ denote  the first $s$ primes, and
let $N(S)$ count the number of 
primitive  solutions  $(X,Y, Z)$ 
to the $S$-unit equation $X+Y=Z.$
Then for each $\epsilon>0$, we have 
$N(S) \gg_{\epsilon}  \exp(s^{\frac{1}{8}-\epsilon})$. 
\end{theorem}

 The approach  in this paper will apply to other  linear  additive problems involving smooth numbers.  For instance, one can treat smooth solutions  of
homogeneous linear ternary Diophantine equations $aX+bY+cZ=0$ 
with arbitrary integer coefficients $(a, b, c)$. 
One may also impose congruence side conditions on the prime factors allowed, for
example smooth solutions with 
all prime factors $p \equiv 1~\bmod 4$. In this situation there may occur local 
congruence obstructions to existence of solutions, and naturally the 
singular series must be modified to take such features into account. 
It would also be of interest to extend the $\XYZ$-conjecture to solutions of
$X+Y=Z$ in  algebraic number fields, or to algebraic function fields 
over finite fields.  Finally, it would be interesting to see if analogues of Waring's problem 
using very smooth numbers could be established.  This has been treated by Harcos \cite{Har97}, 
who obtained unconditional results for Waring's  problem in the smoothness 
range corresponding to the results of Balog and Sark{\" o}zy mentioned earlier.

%
%
%
%
\section{Counting Smooth Solutions: Main Technical Results}
\setcounter{equation}{0}

\noindent Let $x$ and $y$ be large.  Our aim is to count solutions to 
$X+Y=Z$ with $X$, $Y$ and $Z$ being pairwise coprime $y$-smooth integers 
lying below $x$.  We shall simplify the problem by first counting all solutions, primitive and imprimitive, 
to $X+Y=Z$ with $X$, $Y$ and $Z$ being $y$-smooth integers 
up to $x$.  We shall also find it convenient to replace the 
sharp cut-off of being less than $x$ by counting solutions with suitable weights approximating 
the sharp cut-off.   Once this is achieved, a sieve argument will enable us 
to recover primitive solutions from all solutions.

More formally, let $\Phi(x)\in C_{c}^{\infty}(\RR^+)$ be a smooth,
compactly supported, real-valued function on the positive real axis.  We shall 
develop first an asymptotic formula for 
\beql{304}
N(x, y; \Phi)
:= ~~~\sum_{{X, Y, Z \in  \sS(y)}\atop{X+Y=Z}} 
\Phi\Big(\frac{X}{x}\Big) \Phi\Big(\frac{Y}{x}\Big) \Phi\Big(\frac{Z}{x}\Big),  
\eeq
which counts weighted primitive and imprimitive solutions.  \smallskip 

 \begin{theorem}~\label{th21}
{\em (Weighted Smooth Integer Solutions Count)} 
Assume the truth of the GRH.  Let $\Phi$ be a fixed smooth, compactly supported, real 
valued function in  $ C_c^{\infty}(\RR^{+})$.  Let $x$ and $y$ be large, with 
$(\log x)^{8+\delta} \le y \le \exp((\log x)^{\frac{1}{2}-\delta})$ for some $\delta >0$.    Define 
$\kappa$ by the relation $y=(\log x)^{\kappa}$.  
Then, we have 
\beql{552}
N(x, y; \Phi)
=\fS_{\infty}\Big( 1-\frac{1}{\kappa},\Phi\Big)\fS_f\Big(1- \frac{1}{\kappa},y
\Big) \frac{\Psi(x,y)^3}{x} +
O_{ \delta}\Big(  \frac{\Psi(x,y)^3}{x} \frac{\log\log y}{\log y}\Big). 
\eeq
Here the 
``archimedean singular series" $\fS_{\infty}(c, \Phi)$ is 
given by 
\beql{501}
\fS_{\infty}(c,\Phi):= c^3\int_{0}^{\infty} \int_{0}^\infty
\Phi(t_1)\Phi(t_2)\Phi(t_1+t_2)\left( t_1t_2(t_1+t_2)\right)^{c-1} dt_1 dt_2,
\eeq
and the ``non-archimedean singular series" $\fS_f$ is defined by 
\beql{2.3} 
\fS_{f}(c,y) = \prod_{p\le y}
\Big( 1+ \frac{p-1}{p(p^{3c-1}-1)} \Big(\frac{p-p^c}{p-1}\Big)^3 \Big)
\prod_{p>y} 
\Big( 1-\frac{1}{(p-1)^2}\Big).
\eeq 
 \end{theorem}

In our proof, 
it is convenient to restrict $\Phi$ to be compactly supported away 
from $0$.  This restriction prevents us from obtaining an asymptotic formula 
for the number  of nonnegative 
solutions to $X+Y=Z$ with $Z\le x$ and $XYZ$ being $y$-smooth, which 
corresponds to choosing $\Phi$ to be the characteristic function $\chi_{[0,1]}$ of the interval $[0,1]$.
We do expect that the asymptotic formula given in Theorem \ref{th21} will continue to hold in this 
case.   In any event this result  suffices to obtain a lower bound for 
this number of solutions by 
choosing a smooth function $\Phi$ compactly supported inside $\RR^+$ 
which minorizes the characteristic function of $[0,1]$.

%

The compact support of $\Phi(x)$ guarantees that 
the ``weighted archimedean singular series" $\fS_{\infty}(c, \Phi)$ is
defined for all real $c$. In contrast the 
``non-archimedean singular series"  $\fS_{f}(c, y)$ is given
by an Euler product 
that converges to an analytic function for Re$(c)> \frac{1}{3}$ and
diverges at $c= \frac{1}{3}$; here individual terms 
in this Euler product diverge at $c=\frac{1}{3}.$
We observe also that 
$\fS_{f}(c, y)$ has
a phase change in its behavior as $y \to \infty$ at the threshold value $c=\frac{2}{3}$
corresponding to $\kappa=3$. Namely,  we have
\beql{418}
\lim_{y \to \infty} {\fS}_{f}(1-\frac{1}{\kappa}, y)= \left\{ 
\begin{array}{ll} {\fS}_{f}(1-\frac{1}{\kappa}) &~~\mbox{for}~~ \kappa>3,\\
~&~\\
+\infty &~~\mbox{for} ~~ 0<\kappa \le 3,
\end{array}
\right.
\eeq
where for $c> \frac{2}{3}$ we define
\beql{419}
\fS_{f}(c) := \prod_{p}\Big( 1 + 
\frac{p-1}{p(p^{3c-1}-1)} 
\Big(\frac{p - p^c}{p-1}\Big)^3 \Big).
\eeq
The Euler product \eqn{419} converges absolutely to an analytic function of
$c$ on the half-plane
 Re$(c) > \frac{2}{3},$ and  diverges at $c= \frac{2}{3}$. Outside this half-plane,
 on the range $\frac{1}{2} < c  \le \frac{2}{3}$, although one has
 $\fS_{f}(1-\frac{1}{\kappa}, y) \to \infty$ as $y \to \infty$,
 one can show that
 \[
 \fS_{f}(1-\frac{1}{\kappa}, y) \ll \exp ( y^{3/\kappa-1}  ).
 \]
 A consequence is that  for $2 < \kappa \le 3$ one has 
 $ \fS_{f}(1-\frac{1}{\kappa}, (\log H)^{\kappa})\ll H^{\epsilon}$ for any positive $\epsilon$,
 which suggests that the heuristic argument of section 1.2 may
 continue to apply to $N(H, \kappa)$ on this range.

 Using a sieve argument together with Theorem \ref{th21}, we shall treat the weighted count of 
 primitive solutions: 
 \beql{304b}
~~~~~~~N^{\ast} (x, y; \Phi) := \sum_{{X, Y, Z \in  \sS(y)}\atop{X+Y=Z, \gcd(X, Y,Z)=1}} 
\Phi\Big(\frac{X}{x}\Big) \Phi\Big(\frac{Y}{x}\Big) \Phi\Big(\frac{Z}{x}\Big).
\eeq
\smallskip
\begin{theorem}~\label{th121}
{\em (Weighted Primitive Integer Solutions Count)}
Assume the truth of the GRH.
Let $\Phi$ be a fixed smooth, compactly supported, real valued function in $C_c^{\infty}(\RR^{+})$.
Let $x$ and $y$ be large with $(\log x)^{8+\delta}\le y \le \exp((\log x)^{\frac 12-\delta})$.  Define 
$\kappa$ by the relation $y=(\log x)^{\kappa}$.  
Then, we have
\[
N^{\ast}(x, y; \Phi)= 
 \fS_{\infty}\Big(1- \frac{1}{\kappa},\Phi \Big)  \fS_{f}^*\Big(1- \frac{1}{\kappa}, y \Big)
 \frac{\Psi(x, y)^3}{x}
+ O \left( \frac{\Psi(x, y)^3}{x (\log y)^{\frac 14}} \right), 
\]
where the primitive non-archimedean singular series $\fS_f^{*}(c,y) $ was defined in 
(\ref{133}).  
\end{theorem}

 Theorem \ref{th21} and Theorem \ref{th121} together  imply 
 that for nonnegative functions $\Phi$ a smoothed
analogue of Conjecture \ref{cj14} holds  for $\kappa>8$. \smallskip

%
%

\begin{theorem}~\label{th23a}
{\em (Relative Density of Weighted Primitive Smooth Solutions)}
Assume the truth of the GRH.
Then for  any nonnegative function $\Phi(x) \in C_c^{\infty}(\RR_{>0})$
not identically zero,  there holds 
\beql{375}
 \lim_{x \to \infty} \,\frac{ N^{\ast}(x, (\log x)^\kappa; \Phi)}{N(x, (\log x)^{\kappa}; \Phi)} = 
 \, \frac{1}{\zeta(2- \frac{3}{\kappa})},~~~~\mbox{for}~~ \kappa>8.
\eeq
\end{theorem}

 Concerning smaller values of $\kappa$, 
we expect that the asymptotic formulae  given in Theorem \ref{th21} 
 and Theorem \ref{th121} 
continue to hold in the range $\kappa >3$ (so that $c=1-1/\kappa >2/3$).  
If so, then in this range both 
$N(x,y;\Phi)$ and $N^*(x,y;\Phi)$ would be of comparable size, with both being of size 
about $\Psi(x,y)^3/x$, conforming to the heuristic (\ref{109aa}). 
 If $1/2 < c \le \frac{2}{3}$, then 
$\fS_f^*(c,y)$ is of constant size, but $\fS_f(c,y)$ diverges as  $y \to \infty$.  
Thus for  the corresponding 
range $2< \kappa \le 3$, we might still hope that 
the asymptotic formulae of Theorems \ref{th21} and \ref{th121} are true, but 
note that  in this range there are significantly fewer primitive solutions compared to 
imprimitive ones.

The upper bound $y\le \exp((\log x)^{\frac 12-\delta})$ imposed in proving Theorems \ref{th21} and 
\ref{th121} facilitates some of our calculations, but it should be possible to 
remove this condition entirely and obtain similar results.  We have not done 
so, since our interest is in small values of $y$, and moreover in larger ranges 
of $y$ one would expect an unconditional treatment by different means.

Before proceeding to discuss the proofs of our main results stated above, 
we show how the theorems stated in the introduction, as well as Theorem \ref{th23a}, follow from these 
weighted versions.

\begin{proof} [Proof of Theorem~\ref{th13}]  Given any $\epsilon >0$ we may 
construct a smooth function $\Phi_{\epsilon}$ such that $\Phi_{\epsilon}$ 
is smooth and supported on $[\epsilon,1-\epsilon]$, always lies 
between $0$ and $1$, and equals $1$ on the interval $[2\epsilon,1-2\epsilon]$.   
Then $N^*(H,\kappa) \ge N^*(H,(\log H)^{\kappa}; \Phi_\epsilon)$, and we 
may use Theorem \ref{th121} to evaluate the latter quantity.  Since 
$ \fS_{\infty}( c, \Phi_\epsilon) \to  
\fS_{\infty}(c)$ as $\epsilon \to 0$, we deduce Theorem \ref{th13}.
\end{proof}

Theorem \ref{th12} follows immediately from Theorem \ref{th13}.  

\begin{proof}[Proof of Theorem~\ref{th15}]
Let $S$ denote the first $s$ primes,
and choose  $H= \exp( s^{1/8 - \epsilon})$
and $y= p_{s}$.
Then
$(\log H)^{8+\epsilon} = (s^{\frac{1}{8}-\epsilon})^{8 +\epsilon} < s < y$,
so that
$N(S) \ge  N^{*}(H, p_y) \ge N^{*}(H, (\log H)^{8+ \epsilon}).$
Assuming the GRH, Theorem~\ref{th13} gives, 
for sufficiently large $H$,  
$N(S) \ge C_{\epsilon}  H^{2- 3/(8+ \epsilon)} \ge H,$
as asserted. 
\end{proof}


\begin{proof}[Proof of Theorem~\ref{th23a}]
This result is based on the identity of Euler products
 \beql{423}
 \fS_{f}^{\ast}(c)
 := \prod_{p}\Big( \Big( 1 + \frac{p-1}{p(p^{3c-1}-1)} 
\left(\frac{p - p^c}{p-1}\right)^3 \Big) \Big(1- \frac{1}{p^{3c-1}}\Big)\Big)= \frac{1}{\zeta(3c- 1)}\fS_{f}(c).
\eeq
which follows taking $y \to \infty$ in \eqn{133}.
This identity shows that $\fS_{f}(c)$ has a meromorphic continuation to the half-plane
Re$(c)> \frac{1}{2}$, with  its only singularity on
this region being a simple pole at $c=\frac{2}{3}$ having residue $\frac{1}{3}\fS_{f}^{\ast}(\frac{2}{3})$.
In particular, for real $c = 1 -\frac{1}{\kappa}>\frac{2}{3}+\epsilon$ we have
$$
\fS_{f}(c, y) = \fS_f(c) \Big(1+ O_{\epsilon}\Big(\frac{1}{y}\Big)\Big),
$$
and for real $c > \frac{1}{2} +\epsilon$ we have
$$
\fS_{f}^{\ast}(c, y) = \fS_f^{\ast}(c) \Big(1+ O_{\epsilon}\Big(\frac{1}{y}\Big)\Big).
$$

 Substituting these estimates in the main terms of Theorem~\ref{th21} and  Theorem~\ref{th121}
 yields, for $\kappa> 8 +\delta$, the estimate  
\beql{1202b}
N^{\ast}(x, (\log x)^{\kappa}; \Phi)=
\frac{1}{\zeta(2-\frac{3}{\kappa})}N(x, (\log x)^{\kappa}; \Phi))
\left( 1+ O_{\delta}
\left( \frac{1}{(\log\log x)^{\frac{1}{4}}}\right)\right).
\eeq
The positivity hypothesis  on $\Phi$ implies that $N(x, (\log x)^{\kappa}; \Phi)>0$
so we may divide both sides of \eqn{1202b} by it to obtain the ratio estimate \eqn{375}.
\end{proof}

We shall use the Hardy-Littlewood circle method to evaluate $N(x,y;\Phi)$. 
To this end, we  introduce the  weighted exponential sum
\beql{305} 
E(x, y; \alpha) := \sum_{n \in \sS(y)} e(n \alpha) \Phi\Big(\frac{n}{x}\Big),
\eeq
where throughout we use $e(x) := e^{2\pi i x}$.  Then we have 
\beql{306}
N(x, y; \Phi)= 
\int_{0}^{1} E(x,y; \alpha)^2 E(x,y; - \alpha) d \alpha,
\eeq
because in multiplying out the exponential sums in the integral,
only terms $(n_1, n_2, n_3)$ with $n_1+n_2-n_3=0$ contribute.
The crux of the problem then is to understand the weighted exponential 
sum $E(x,y;\alpha)$.  

To do this, we show how to express the term $e(n\alpha) \Phi(n/x)$ 
in terms of sums over multiplicative Dirichlet characters to a 
certain modulus and integrals of $n^{it}$ over $t$ in a 
certain range.  This is carried out precisely in Section 3, but 
the idea is implicit in the original `Partitio Numerorum'  papers of Hardy and Littlewood 
(\cite{HL23}, \cite{HL24}) where they dealt with the ternary Goldbach problem 
assuming a weaker form of GRH.   We hope that the explicit form that we 
give may be useful in other contexts.  

The decomposition of $e(n\alpha) \Phi(n/x)$ in terms of multiplicative characters 
converts the problem of understanding $E(x,y;\alpha)$ to one of understanding 
$\sum_{n\in \sS(y)} \chi(n) n^{-it} \Phi(n/x)$ for suitable Dirichlet characters $\chi$ and 
suitable real numbers $t$.  We establish, on GRH, that such sums are small 
unless $\chi$ happens to be the principal character, and $|t|$ is small.  The key 
step in achieving this is to bound partial 
Euler products $L(s,\chi;y) =\prod_{p\le y}(1-\chi(p)/p^{s})^{-1}$ on GRH.  The 
bounds for these partial Euler products that we establish are analogous to the 
Lindel{\" o}f bounds for Dirichlet $L$-functions, and the (familiar) argument 
is described in \S 5.  In this fashion, we are able to understand 
conditionally the weighted exponential sum $E(x,y;\alpha)$, and in \S 6 we 
establish the following Theorem.  


\begin{theorem}\label{th2.3} Assume the truth of the GRH.  Let $\delta>0$ be any 
fixed real number.  Let $x$ and $y$ be 
large with $(\log x)^{2+\delta} \le y\le \exp((\log x)^{\frac 12-\delta})$, and let 
$\kappa$ be defined by $y=(\log x)^{\kappa}$.  Let $\alpha\in [0,1]$ be a 
real number with $\alpha=a/q+\gamma$ where $q\le \sqrt{x}$, $(a,q)=1$, and 
$|\gamma|\le 1/(q\sqrt{x})$. 

(1)  If $|\gamma|\ge x^{\delta-1}$ then we have, for any fixed $\epsilon>0$, 
$$ 
E(x,y;\alpha) \ll x^{\frac{3}{4}+\epsilon}. 
$$ 

(2) If $|\gamma|\le x^{\delta-1}$ then we have, writing $q=q_0q_1$ with 
$q_0 \in {\sS(y)}$ and all prime factors of $q_1$ being bigger than $y$, and writing ${c_0}=1-1/\kappa$, 
for any fixed $\epsilon >0$, 
\begin{eqnarray*}
E(x,y;\alpha)& = &\frac{\mu(q_1)}{\phi(q_1)} 
\frac{1}{q_0^{c_0}}\prod_{p|q_0} \Big(1- \frac{p^{c_0}-1}{p-1}\Big) 
\Big( c_0\int_0^{\infty} 
\Phi(w) e(\gamma xw) w^{c_0-1} dw\Big) \Psi(x,y)  \\
&&+ O_{\epsilon}(x^{\frac{3}{4}+\epsilon}) + O_{\epsilon}\Big( \frac{\Psi(x,y)q_0^{-c_0+\epsilon} q_1^{-1+\epsilon}}{(1+|\gamma|x)^2} \frac{(\log \log y)}{\log y} \Big). \\
\end{eqnarray*}
\end{theorem}

The proof supposes $y \ge (\log x)^{2+\delta}$, 
but the result only gives a nontrivial estimate for somewhat larger $y$ because
for $\kappa \le 4$   one 
has the trivial estimate
\[
|E(x, y; \alpha)| 
\ll \Psi(x, y) 
\ll x^{\frac{3}{4} + \epsilon}.
\]

Note that by Dirichlet's theorem on Diophantine approximation one can always 
find $q\le \sqrt{x}$, and $(a,q)=1$ with $|\alpha-a/q| \le 1/(q\sqrt{x})$.   
Theorem \ref{th2.3} then shows that $E(x,y;\alpha)$ is small unless 
$q$ is small and $|\gamma|$ is small.  In other words, Theorem \ref{th2.3} 
can be used to estimate $E(x,y;\alpha)$ on the {\sl minor arcs} where 
$\alpha$ is not near a rational number with small denominator, and it 
also furnishes an asymptotic formula for our exponential sum 
when $\alpha$ lies on a {\sl major arc}.   We shall define the 
major and minor arcs more precisely in \S 7, where we use 
the results leading to Theorem \ref{th2.3} to complete the 
proof of Theorem \ref{th21}.

 We should point out that the exponential sum $\sum_{n\le x, n\in {\sS(y)}} e(n\alpha)$ 
 has been studied unconditionally be several authors, see  de la Bret\'{e}che 
 (\cite{dlB98}, \cite{dlB99}), de la Bret\'{e}che and Tenenbaum (\cite{BT04}, \cite{BT05}, \cite{BT07}), 
 and de la Bret{\' e}che and Granville \cite{BG09}.    Our work gives better 
 estimates, and holds in wider ranges of $y$, but on the other hand it relies on the truth of the GRH.

In the range of interest to us, namely $y$ being a power of $\log x$, it is a 
delicate problem even to count the number of $y$-smooth integers up to $x$.  
One important ingredient in our work is the saddle-point method developed 
by Hildebrand and Tenenbaum \cite{HT86} which provides an asymptotic formula for $\Psi(x,y)$ 
in such ranges.  In \S 4, we survey briefly results on $\Psi(x,y)$ and extract the key results from the 
Hildebrand-Tenenbaum approach that we require.

Finally, in \S 8 we give a sieve argument that allows us to pass from all 
the solutions counted in Theorem \ref{th21} to only the primitive solutions 
counted in Theorem \ref{th121}.

%
%
\section{Multiplicative Character Decomposition}
\setcounter{equation}{0}

\noindent In this section we show how to express $e(n\alpha)\Phi(n/x)$ for $\alpha \in [0,1]$  
in terms of sums over multiplicative Dirichlet characters to a certain 
modulus and integrals of $n^{it}$ over $t$ in a certain range.   To achieve this we 
write $\alpha=a/q+\gamma$ with $(a,q)=1$, and then our decomposition will involve Dirichlet characters 
$\!\!\!\!\bmod q$ and functions $n^{it}$ where $t$ is roughly of size $1+|\gamma|x$.   
When $\alpha=a/q$ is a rational number, this is the familiar technique of expressing 
additive characters in terms of multiplicative characters, and our decomposition 
may be viewed as an extension of that method.

Let us first recall the decomposition of the additive character $e(an/q)$ in 
terms of multiplicative characters.   For a Dirichlet character $\chi \imod q$, not 
necessarily primitive, recall that the Gauss sum is defined by 
$\tau(\chi) = \sum_{b \imod q} \chi(b) e(b/q)$.  \smallskip

%
%

\begin{lemma}~\label{le31}  Let $a/q$ be a rational number with $(a,q)=1$.  

(1)  Let $n$ be an integer, and suppose that $(n,q)=d$.  Then with $n=md$ we have  
\beql{311a}
e\Big(\frac{an}{q}\Big)= e\Big(\frac{ma}{q/d}\Big) = \frac{1}{\phi(q/d)}
\sum_{\chi \imod {q/d}} \tau(\bar{\chi}) \chi(ma).
\eeq

(2) One has
\beql{312a}
\frac{1}{\phi(q/d)^2} \sum_{\chi \imod{q/d}} |\tau(\chi)|^2= 1.
\eeq
\end{lemma}

\proof  Both relations follow readily  from the definition of 
the Gauss sum and the orthogonality relations for the Dirichlet characters $\imod{q/d}$.  
\endproof
\smallskip

%
%

\begin{lemma}~\label{le31a} (Gauss sum estimate)  If $\chi \imod q$ is 
primitive then $|\tau(\chi)|=\sqrt{q}$.  If $\chi$ is induced by the primitive character $\chi' \imod{q'}$ then
\beql{313a}
\tau(\chi) = \mu\Big(\frac{q}{q'}\Big) \chi'\Big(\frac{q}{q'}\Big) \tau(\chi'),
\eeq
where $\mu(n)$ is the M\"{o}bius function, and so in this case $|\tau(\chi)| \le \sqrt{q^{\prime}}\le 
\sqrt{q}$. 
\end{lemma} 

\proof
This is standard; see, for example Lemma 4.1 of 
Granville and \\ Soundararajan \cite{GS07}. 
\endproof

Now we turn to $e(n\gamma) \Phi(n/x)$ which we would like to express as 
an integral involving the multiplicative functions $n^{it}$.  To do this, we define 
\beql{331} 
\check{\Phi}( s, \lambda)
:= \int_{0}^{\infty} \Phi(w)e(\lambda w) w^{s-1}dw.
\eeq
Since $\Phi$ has compact support inside $(0,\infty)$ the integral above makes 
sense for all complex numbers $\lambda$ and $s$, but we shall be only interested in 
the case $\lambda$ real.  Note that $e(\lambda w)$ has the structure of an additive character
while $w^{s}$ has the structure of a multiplicative character so that the transform ${\check \Phi}(s,
\lambda)$ plays a role analogous to the Gauss sum.

We begin by showing that ${\check \Phi}(s,\lambda)$ is small unless $1+|\lambda|$ and 
$1+|s|$ are of roughly the same size.  \smallskip

%
%
\begin{lemma}~\label{le33}  Let $\Phi$ be a smooth function, compactly supported
in $(0, \infty)$.
Let $\lambda$ be real and suppose Re$(s)\ge 1/4$.  Then 
for  any non-negative integer $k$ we have 
\beql{341}
|\check{\Phi}(s, \lambda)| \ll_{k, \Phi} \min \Big( \Big(\frac{1+|\lambda|}{|s|}\Big)^k, 
\Big(\frac{1+|s|}{|\lambda|}\Big)^k \Big).
\eeq
\end{lemma} 

\proof  We integrate by parts $k$ times, and can do this in two ways either 
using the pair of functions $\Phi(w)e(\lambda w)$ and $w^{s-1}$, or using the 
pair of functions $\Phi(w)w^{s-1}$ and $e(\lambda w) $.   Integrating by parts $k$ 
times using the first pair we obtain 
$$ 
{\check \Phi}(s,\lambda) =  (-1)^k \int_{0}^{\infty} 
\frac{d^k}{dw^k} \Big(\Phi(w)e(\lambda w)\Big) \frac{w^{s+k-1} }{s(s+1)\cdots(s+k-1)}dw.
$$
Since 
$$ 
\frac{d^k}{dw^k} \Big(\Phi(w)e(\lambda w)\Big) =\sum_{j=0}^{k} {k \choose j} \Phi^{(j)}(w) 
(2\pi i \lambda)^{k-j} e(\lambda w) \ll 2^k \sum_{j=0}^{k} |\Phi^{(j)}(w)| (2\pi |\lambda|)^{k-j}, 
$$ 
we conclude that 
$$ 
{\check \Phi}(s,\lambda) \ll_k \frac{1}{|s|^k} \sum_{j=0}^{k} |\lambda|^{k-j} \int_0^{\infty} |\Phi^{(j)}(w) w^{s+k-1}dw| \ll_{k, \Phi} \Big(\frac{1+|\lambda|}{|s|}\Big)^k.
$$ 

On the other hand, integrating by parts using the second pair we obtain 
$$ 
{\check \Phi}(s,\lambda) = (-1)^k \int_0^{\infty} \frac{d^k}{dw^k} \Big( \Phi(w) w^{s-1} \Big) 
\frac{e(\lambda w)}{(2\pi i \lambda)^k} dw.
$$ 
Since 
\begin{eqnarray*}
\frac{d^k}{dw^k} \Big(\Phi(w) w^{s-1} \Big) &=& \sum_{j=0}^k {k \choose j} {\Phi}^{(j)}(w) 
(s-1)\cdot (s-2) \cdots (s-(k-j)) w^{s-1-(k-j)} \\
&\ll_k& \sum_{j=0}^{k} |{\Phi}^{(j)}(w)| |s|^{k-j} |w|^{s-1-(k-j)},\\
\end{eqnarray*}
we conclude that 
$$ 
{\check \Phi}(s,\lambda) \ll_k \frac{1}{|\lambda|^k}\sum_{j=0}^{k} |s|^{k-j} \int_0^{\infty} 
|{\Phi}^{(j)}(w) w^{s-1-(k-j)} dw| \ll_{k,\Phi} \Big(\frac{1+|s|}{|\lambda|}\Big)^k.
$$
\endproof

Now we prove an analog of Lemma \ref{le31} for $e(n\gamma) \Phi(n/x)$.\smallskip

%
%
\begin{lemma}~\label{le32}  Let $\Phi$ be a smooth function compactly supported 
in $(0,\infty)$.  

(1) For $n \in \ZZ$, we have  for any positive $c={\rm Re}(s)$, 
\beql{318a}
e(n \gamma) \Phi\Big(\frac{n}{x}\Big) = \frac{1}{2\pi i}\int_{c - i \infty}^{c+i\infty}
\check{\Phi}( s, \gamma x) \left(\frac{x}{n}\right)^{s} ds.
\eeq

(2) Furthermore 
\beql{319a}
\frac{1}{2\pi} \int_{-\infty}^{\infty} | \check{\Phi}(c+it, \gamma x)|^2 dt=
\int_{-\infty}^{\infty} | \Phi(e^u) e(\gamma x e^u) e^{cu}|^2 du.
\eeq
\end{lemma}

\proof From the definition of ${\check \Phi}$ and Mellin inversion, we 
obtain for $w>0$,
\[
e(\lambda w) \Phi(w) = \frac{1}{2\pi i} \int_{c - i \infty}^{c+i\infty}
\check{\Phi}(s, \lambda) w^{-s} ds.
\]
We obtain \eqn{318a} on taking $w=\frac{n}{x}$, and $ \lambda= \gamma x$.

Take   $s=c+it$ in the definition of ${\check \Phi}$, 
and change variables $w=e^{u}$.  Thus 
\begin{eqnarray*}
\check{\Phi}(c+it, \lambda)  = 
\int_{0}^{\infty} \Phi(w) e(\lambda w) w^{c+it} \frac{dw}{w}  =  
\int_{-\infty}^{\infty} \Phi(e^u) e(\lambda e^u) e^{c u + i tu} du,
\end{eqnarray*}
and we recognize that $\check{\Phi}(c+it, \lambda)$, viewed as a function of $t$ with $c$ and $\lambda$ fixed, is the Fourier transform of $\Phi(e^u)e(\lambda e^u) e^{cu}$.  
Now Plancherel's theorem gives 
\[
\frac{1}{2 \pi} \int_{-\infty}^{\infty} |\check{\Phi}(c+it, \lambda)|^2 ct =
\int_{-\infty}^{\infty} | \Phi(e^u)e(\lambda e^u) e^{cu}|^2 du.
\]
which, with $\lambda=\gamma x$,  yields \eqn{319a}. 
\endproof

Using the method of stationary phase, we can show that $|{\check \Phi}(c+it,\lambda)| 
\ll (1+|\lambda|)^{-\frac 12}$ and this bound is an analog of the bound $|\tau(\chi)|\le \sqrt{q}$ 
for Gauss sums.  In our applications an $L^1$ version of this bound is sufficient, and 
we next derive such a bound from the $L^2$ estimate above.  \smallskip

%
%

\begin{lemma}~\label{le34}
Let $\lambda$ be real and suppose that $c\ge \frac 14$.  For any $\delta\ge 0$ and any 
$\epsilon>0$,  we have
\beql{331a}
\int_{-\infty}^{\infty} |\check{\Phi}(c+it, \lambda)| (1+|t|)^{\delta} 
dt \ll_{\Phi, c,\epsilon} (1+ |\lambda|)^{\frac{1}{2} + \delta+\epsilon}.
\eeq
\end{lemma}

\proof Let $\epsilon>0$ be given.  Consider first the range when $|t|>(1+|\lambda|)^{1+\epsilon}$.  
Using Lemma \ref{le33} we find that for any integer $k \ge 2$ 
\begin{eqnarray*}
\int_{|t|>(1+|\lambda|)^{1+\epsilon}} |{\check \Phi}(c+it,\lambda)| (1+|t|)^{\delta} dt 
&\ll_{k, \Phi}& \int_{|t| > (1+|\lambda|)^{1+\epsilon}} \Big( \frac{1+|\lambda|}{1+|t|}\Big)^k (1+|t|)^{\delta} dt 
\\
&\ll_{k, \Phi}& (1+|\lambda|)^{k -(k-\delta +1)(1+\epsilon)}. 
\\
\end{eqnarray*}
Choosing $k$ suitably large, this contribution is $\ll_{\Phi, \epsilon} 1$.  

Now consider the range $|t| \le (1+|\lambda|)^{1+\epsilon}$.  Note that 
$$
\int_{|t|\le (1+|\lambda|)^{1+\epsilon}} |{\check \Phi}(c+it, \lambda)| (1+|t|)^{\delta} dt 
\ll (1+|\lambda|)^{\delta(1+\epsilon)} \int_{|t|\le (1+\lambda)^{1+\epsilon}} |{\check \Phi}(c+it,\lambda)| dt,
$$ 
and using Cauchy-Schwarz 
we see that 
\begin{eqnarray*}
\int_{|t| \le (1+ \lambda)^{1+ \epsilon}} |\check{\Phi}(c+it, \lambda)| dt  &\le &
 \left( \int_{|t| \le (1+ |\lambda |)^{1+\epsilon}}  1 ~dt \right)^{\frac{1}{2}} 
\left( \int_{|t| \le (1+|\lambda |)^{1+\epsilon}} |\check{\Phi}(c+it, \lambda)|^2 dt \right)^{\frac{1}{2}}\\
&\le&  ( 1 + |\lambda|)^{\frac{1}{2}+\frac{1}{2}\epsilon}
\left( \int_{-\infty}^{\infty} |\check{\Phi}(c+it, \lambda)|^2 dt \right)^{\frac{1}{2}}\\
& \ll_{\Phi, c} & ( 1 + |\lambda|)^{\frac{1}{2} + \frac{1}{2}\epsilon},
\end{eqnarray*}
upon using the Plancherel formula from Lemma \ref{le32}(2). 
The Lemma follows.
\endproof

Combining the formulas \eqref{311a} and \eqref{318a} for $\alpha= \frac{a}{q} + \gamma$, for
$n \ge 1$ with $(n, q)=d$ we obtain
\beql{491a}
e(n \alpha) \Phi(\frac{n}{x}) =\Big( \frac{1}{\phi(q/d)} \sum_{\chi \imod{\frac{q}{d}}}
\tau(\bar{\chi}) \chi\Big( \frac{na}{d}\Big) \Big) \Big( \frac{1}{2 \pi i} \int_{-\infty}^{\infty}
\check{\Phi}( s, \gamma x) \Big(\frac{x}{n}\Big)^s dx\Big).
\eeq

Lemma~\ref{le31} and Lemma~\ref{le32} exhibit parallels between
the Dirichlet characters $\chi(n) \imod{q}$ (the $q$-aspect) 
and the continuous family of characters $\chi_t(n)=n^{it}$ (the $t$-aspect). 
Part (1) of each lemma expresses the (weighted) additive character 
in terms of multiplicative characters.   
Gauss sums appear explicitly in Lemma~\ref{le31}, while 
in Lemma~\ref{le32} the 
 function  $\check{\Phi}(c+it, \lambda)$ plays a role analogous to a Gauss sum, as it is
a weighted convolution of an additive quasicharacter 
specified by the parameter $\lambda$
against  a multiplicative quasicharacter 
by $\chi_{c+it}(n) =n^{c+it}$. The weight function $\Phi(x)$ limits
the range sampled, and  
 Lemma~\ref{le33} gives bounds on the size of this function.
  Part (2) of each lemma expresses an $L^2$-orthogonality relation. 
These orthogonality relations imply that the change of basis
 to multiplicative characters 
loses essentially nothing in the $L^2$-sense. However
in our application, the $L^1$-norm is more relevant,  
and there is a loss in moving from additive to multiplicative characters.
This is quantified in the square root losses in the both $q$ and $t$ aspects
paralleled  in the  ``Gauss sum" type
estimates  in Lemma~\ref{le31a} and Lemma~\ref{le34}, respectively.

\begin{remark} 
{\rm In Theorem \ref{th21} we would like to subsititute the sharp cutoff
 weight function $\Phi(x) = \chi_{[0,1]}(x)$, but  it   is neither compactly supported nor
continuous on $\RR_{>0}$, and we only obtain a lower bound \eqref{552}
rather than the expected asymptotic formula.
 Here we note in passing that the transform $\check{\Phi}(s, \lambda)$ given in \eqref{331}
 is an interesting special function. Namely, for $Re(\lambda)<0$, we have
\beql{381}
\check{\Phi}(s, \lambda) =\int_{0}^{1} e^{\lambda x} x^{s-1}dx = (-\lambda)^{-s} \gamma(s, -\lambda),
\eeq
where $\gamma(s, z)= \int_{0}^{z} e^{-u} u^{s-1} du$ is the incomplete gamma function.
The incomplete gamma function is related  to Kummer's  
confluent hypergeometric function
$$
M(a, b, z) := {}_1F_{1} (a, b; z) = 1 + \frac{a}{b} \frac{z}{1!} + \frac{a(a+1)}{b(b+1)} \frac{z^2}{2!} +
\frac{a(a+1)(a+2)}{b(b+1)(b+2)}\frac{z^3}{3!} +  \cdots,
$$
by special function formulas (see Chapter 13 of  \cite{AS72}) which yield
\beql{382}
s(-\lambda)^{-s} \gamma(s,  -\lambda)= M(s, s+1, \lambda) =e^{\lambda}M(1, s+1, \lambda).
\eeq 
The last equality is a special case of Kummer's transformation
$M(a, b, z)= e^z M(b-a, b, z).$ The known analytic properties of the
function $M(a, b, z)$  (in three complex variables)  give an analytic continuation of 
$\frac{1}{\Gamma(s)} \check{\Phi}(s, \lambda)$ to  an entire function of two complex variables.
It follows that $\check{\Phi}(s, \lambda)$ has no singularities   in the $\lambda$-variable,  but 
for generic $\lambda$ it has
 simple poles in the $s$-variable  at the nonpositive integers. }
\end{remark}

\section{A brief survey of results on $\Psi(x,y)$}
 \setcounter{equation}{0}

\noindent In this section we collect together several results  
on estimates for $\Psi(x,y)$.  A comprehensive survey of this topic is 
given by Hildebrand and Tenenbaum \cite{HT93}, and we 
give here a very brief description of the salient points.  

When $y$ is not too small in relation to $x$, then on writing $y=x^{\frac 1u}$, 
we have that $\Psi(x, y) \sim  x \rho(u)$ where $\rho$ is the Dickman 
function which is defined by $\rho(u)= 1$ for $0\le u\le 1$, and for $u\ge 1$ is 
defined by the differential-difference equation $u\rho^{\prime}(u) = -\rho(u-1)$.  
The most precise version of this result is due to Hildebrand \cite{Hi86b} 
who showed that for all large  $x$ and $y\ge \exp( (\log\log x)^{5/3 + \epsilon})$
we have 
\beql{121a}
\Psi(x, y) = x \rho(u) \Big( 1+ O_{\epsilon}\Big( \frac{u\log(u+1)}{\log x}\Big) \Big).
\eeq 

Here we are 
particularly interested in the range when $y$ is a power of $\log x$.  This 
is the relevant range for our main results, but it lies outside the 
range covered by Hildebrand's (\ref{121a}).   Indeed in this range, the 
behavior of $\Psi(x,y)$ is known to be sensitive to the fine distribution of primes and
 location of the zeros of $\zeta(s)$.
In 1984 Hildebrand \cite{Hi84} showed that the Riemann hypothesis is equivalent
to the assertion that for each $\epsilon >0$ and $1 \le u \le y^{1/2 - \epsilon}$
there is a uniform estimate
\beql{122a}
\Psi(x,y) =  x \rho(u) \exp( O_{\epsilon} ( y^{\epsilon} ) ).
\eeq
Moreover, assuming the Riemann hypothesis, he showed that 
for each $\epsilon >0$ and $1 \le u \le y^{1/2 - \epsilon}$
the stronger uniform estimate
\beql{122b}
\Psi(x,y) =   x \rho(u) \exp \Big( O_{\epsilon} \Big( \frac{\log (u+1)}{\log y} \Big) \Big)
\eeq
holds. 
On choosing $y = (\log x)^{\alpha}$ for $\alpha > 2$, this latter estimate yields 
\beql{123c}
\Psi(x, (\log x)^{\alpha} ) \asymp  x \rho(u),
\eeq
which  provides only an order of magnitude estimate for the size of $\Psi(x,y)$. 
Furthermore if
the Riemann hypothesis is false then $\Psi(x,y)$ must sometimes
exhibit large oscillations away from the value $x \rho(u)$ for some $(x,y)$ in
these ranges.  In 1986 Hildebrand \cite{Hi86} obtained further results 
indicating  that when
$y < (\log x)^{2- \epsilon}$ one should not expect  any smooth asymptotic
formula for $\Psi(x,y)$ in terms of the $y$-variable to hold.

Since we assume GRH in this paper, we may access these conditional 
results of Hildebrand. However a less explicit asymptotic formula for $\Psi(x,y)$ 
developed by  Hildebrand and Tenenbaum \cite{HT86} is more useful for us. 
Before discussing the results from their saddle point method, we 
note a useful, and uniform, elementary asymptotic for $\log \Psi(x,y)$; see 
Theorem 1.4 of \cite{HT93}.    
Uniformly for all $x\ge y\ge 2$ there holds 
$$ 
\log \Psi(x,y) = \Big(\frac{\log x}{\log y} \log \Big( 1+\frac{y}{\log x}\Big) 
+ \frac{y}{\log y} \log \Big( 1+ \frac{\log x}{y}\Big)\Big) \Big( 1+ 
O\Big(\frac{1}{\log y}+\frac{1}{\log \log x}\Big)\Big).
$$ 
If $y= (\log x)^{\alpha}$, with $\alpha \ge 1$ then it follows 
that 
\beql{S801}
\Psi(x,y) = x^{1-\frac{1}{\alpha}} \exp\Big(O\Big(\frac{\log x}{\log \log x}\Big)\Big).
\eeq

We define
$$
\zeta(s; y) := \sum_{n \in \sS(y)} n^{-s}=
\prod_{p \le y} \Big(1- \frac{1}{p^{s}}\Big)^{-1},
$$
 and by Perron's formula we may write, for any $c>0$, 
\beql{S802}
\Psi(x,y) =\frac{1}{2\pi i } \int_{c-i\infty}^{c+i\infty} \zeta(s;y)x^s \frac{ds}{s}.
\eeq
The method of Hildebrand and Tenenbaum makes a careful choice 
for the line of integration $(c)$.   Precisely, they choose $c$ such 
that the quantity $x^{\sigma} \zeta(\sigma;y)$ is minimized over 
all $0 <\sigma\le \infty$.   With a little calculus, this quantity is minimized when 
$c=c(x,y)$ is the unique solution to 
\beql{905}
-\phi_1(c;y):= - \frac{d}{dc} \log \zeta(c; y) = \sum_{p \le y} \frac{\log p}{p^c -1}  =  \log x, 
\eeq
where $\phi_j(c;y)$ denotes the $j$-th derivative with respect to $s$ of $\log \zeta(s;y)$.    
The quantity\footnote{Hildebrand and Tenenbaum denote this
quantity $\alpha(x,y)$ and abbreviate it to $\alpha$.}
$c(x,y)$ is a saddle-point for the function $x^s \zeta(s; y)$ in the sense that 
$|x^s \zeta(s;y)|$ is minimized over real values of $s \in (0,\infty)$, but is 
maximized over values $s=c+it$ for $t\in \RR$.   With this 
choice for the line of integration, Hildebrand and Tenenbaum found that the 
integral in (\ref{S802}) is dominated by the portion of the 
integral near the real axis, and were able to evaluate this contribution.  We now quote 
their result, see Theorem 1 of \cite {HT86}.  \smallskip
%

\begin{theorem}~\label{th91}
{\em (Hildebrand-Tenenbaum)} 
We have uniformly for $x \ge y \ge 2$, 
\beql{907} 
\Psi(x,y) = \frac{x^c \zeta(c;y)}{c\sqrt{2\pi \phi_2(c,y)}}
\left( 1+ O\left( \frac{1}{u}+\frac{\log y}{y}\right) \right),
\eeq
in which $c=c(x,y)$, and $y=x^{\frac 1u}$. 
\end{theorem}
 
 The following result, Theorem 2 of \cite{HT86}, concerns the size of $c(x,y)$ and of the 
denominator in \eqn{907}, involving
$$
\phi_2(c;y) = \frac{d^2}{dc^2} \log \zeta(c;y) =  \sum_{p \le y} \frac{p^c (\log p)^2}{(p^c -1)^2}.
$$
\smallskip

%
\begin{theorem}~\label{th92}
{\em (Hildebrand-Tenenbaum)} We have uniformly for $x \ge y \ge 2$,
\beql{908}
c(x,y) = \frac{\log \left(1+ \frac{y}{\log x}\right)}{\log y}\left( 1+
O\left(\frac{\log \log (1+y)}{\log y}\right) \right),
\eeq
and
\beql{909}
\phi_2(c(x,y), y) =\left( 1 + \frac{\log x}{y}\right) \log x
\cdot \log y 
\left( 1+ O \left( \frac{1}{\log (1+u)}+ \frac{1}{\log y}\right)\right).
\eeq
\end{theorem}

An immediate consequence of \eqn{908}  is  that for  fixed $\delta>0$, and 
$y =(\log x)^{\kappa}$ with $\kappa \ge 1+\delta$ we have 
\beql{819a}
c(x,y) = 1- \frac{1}{\kappa} + 
O_{\delta}\Big(\frac{\log \log y}{\log y}\Big).
\eeq

While the asymptotic in Theorem \ref{th91} may be a little 
difficult to parse, it provides an elegant and useful means 
of obtaining the ``local behavior" of $\Psi(x,y)$, given as
follows -- see Theorem 3 of  \cite{HT86}.  \smallskip
 
%
\begin{theorem}~\label{th93}
{\em (Hildebrand-Tenenbaum)} 
We have uniformly for $x \ge y \ge 2$ and $1 \le k \le y$,
\beql{905a}
\Psi(kx, y) = \Psi(x,y) k^{c(x,y)}
\left( 1+ O\left(\frac{\log y}{\log x} + \frac{\log y}{y}\right)\right).
\eeq
\end{theorem}
 
This result can be used to show that 
the  behavior of  $\Psi(x, y)$ 
with $y= (\log x)^{\kappa}$ changes
qualitatively at $\kappa=1$, having a ``phase
transition'' there.  As $x\to \infty$, Theorem~\ref{th93} implies that 
when $\kappa \le 1$ one has   
\[
\Psi(kx, y) = (1+o(1))\Psi(x, y), 
\]
whereas for $\kappa >1$ one has 
\[
\Psi(kx,y) = \left(k^{1-\frac{1}{\kappa}} +o(1)\right)\Psi(x,y).
\] 
 
For later use, we state three estimates of Hildebrand and
Tenenbaum (restricted to the range $y\ge \log x$) as lemmas.  \smallskip
 
%
\begin{lemma}~\label{le91}
{\em (Hildebrand and Tenenbaum)}
Let $x$ and $y$ be large with $y\ge \log x$, and let $s = c+i \tau$ with
$c=c(x,y)$ and real $\tau$.  Uniformly in the region $1/\log y \le |\tau|\le y$ we have 
\beql{912}
\Big|\frac{\zeta(s; y)}{\zeta(c;y)}\Big| \ll 
\exp \Big( - c_0 \frac{u \tau^2}{(1-c)^2 + \tau^2} \Big).
\eeq
\end{lemma}
 
\proof This is a special case of Lemma 8 of \cite{HT86}.
\endproof 
\smallskip

%
\begin{lemma}~\label{le92}
{\em (Hildebrand and Tenenbaum)}
Let $0< \beta < 1$ be fixed.
Then uniformly for $x\ge y \ge 2$, 
\begin{eqnarray}~\label{921}
\Psi(x,y) & = & 
\frac{1}{2\pi i} \int_{c - \frac{i}{\log y}}^{c+\frac{i}{\log y}}
\zeta(s;y) \frac{x^s}{s} ds \nonumber \\
&& ~~~+ O_{\beta} \Big( x^{c}\zeta(c,y) 
\Big( \exp ( -(\log y)^{\frac{3}{2} - \beta}) +
\exp \Big(-c_6 \frac{u}{(\log 2u)^2}\Big) \Big)\Big),
\end{eqnarray}
with $c=c(x,y)$, and $c_6>0$ an absolute  constant.
\end{lemma}

\proof This is Lemma 10 of \cite{HT86}.
\endproof
\smallskip

%
\begin{lemma}~\label{le93}
{\em (Hildebrand and Tenenbaum)}
If $x$ and $y$ are large, and $y\ge\log x$, 
\beql{931}
\frac{1}{2\pi i} \int_{c - \frac{i}{\log y}}^{c+\frac{i}{\log y}}
\zeta(s;y) \frac{x^s}{s} ds =
\frac{x^{c} \zeta(c;y)}{c \sqrt{2\pi \phi_2(c; y)}} 
\left( 1 + O \left( \frac{1}{u}\right) \right),
\eeq
with $c=c(x,y)$.
 Moreover, the same estimate holds for
\beql{932}
\frac{1}{2\pi} \int_{c - \frac{i}{\log y}}^{c+i\frac{i}{\log y}}
|\zeta(s;y) \frac{x^s}{s}| |ds| =
\frac{x^{c} \zeta(c;y)}{c \sqrt{2\pi \phi_2(c; y)}} 
\left( 1 + O \left( \frac{1}{u}\right) \right).
\eeq
\end{lemma}

\proof This is Lemma 11 of \cite{HT86}, restricted to the range $y\ge \log x$.
\endproof
\smallskip

The agreement in size of the integral \eqn{931} with the absolute value estimate \eqn{932},
is a key feature of the integral being at the saddle point.
We remark that Lemma \ref{le92} and Lemma \ref{le93} are
major  ingredients used by Hildebrand and Tenenbaum  in proving  Theorem \ref{th91}.

%
%
%
%

\section{Bounds for partial $L$-functions on GRH} 
\setcounter{equation}{0}

\noindent It is well-known that the generalized Riemann hypothesis 
implies the generalized Lindel{\" o}f hypothesis:  If $\chi\imod q$ is 
a primitive character and $s$ is a complex number with Re$(s)\ge 1/2$, 
then for any $\epsilon >0$ we have $|L(s, \chi)| 
\ll_{\epsilon} (q |s|)^{\epsilon}$.   Our aim in this section is to 
establish a corresponding conditional estimate for the partial Euler products 
$$L(s,\chi; y):= \prod_{p\le y}\Big( 1- \chi(p)p^{-s}\Big)^{-1}.$$
\smallskip

%
%
\begin{prop}~\label{le43}  Assume  the truth of the GRH.  Let $\chi \imod{q}$ be a 
primitive Dirichlet character.       
For any $\epsilon >0$, and $s$ a complex number with Re$(s)=\sigma\ge 1/2+\epsilon$, we 
have  
\beql{431}
|L(s, \chi; y)| \ll_{\epsilon} (q|s|)^{\epsilon}.
\eeq
For the trivial character we have, with $\sigma ={\rm Re}(s)\ge 1/2+\epsilon$,  
\beql{432}
|\zeta(s; y)| \ll_{\epsilon} \exp\Big(\frac{y^{1- \sigma}}{(1+|t|)\log y} \Big) |s|^{\epsilon}.
\eeq
\end{prop}
 
 We shall prove Proposition \ref{le43} by developing conditional 
 estimates for $\sum_{n\le u} \Lambda(n)\chi(n) n^{-it}$.  These estimates 
 follow from standard ``explicit formula" arguments connecting such prime sums 
 with zeros of the corresponding $L$-function, and we shall be brief 
 in sketching their proofs.  \smallskip

%
%
\begin{lemma}~\label{le41}  Let $\chi \imod q$ be a primitive Dirichlet character, 
and let $t$ be a real number.   Let $\rho=\beta+i\gamma$ denote a typical 
zero of the Dirichlet $L$-function $L(s,\chi)$.   Let $\delta(\chi)=1$ if $q=1$ and 
$\chi$ is the principal character, and $\delta(\chi)=0$ otherwise.  
Then for $u\ge 2$ and any parameter $T \ge 2$ we have 
\begin{eqnarray*}
\sum_{n\le u} \Lambda(n) \chi(n) n^{-it} &=& \delta(\chi)\frac{u^{1-it}}{1-it} - 
 \sum_{ {{0<\beta <1}\atop{|\gamma-t | \le T}} } 
\frac{u^{\rho- it}}{\rho - it} \\
&& + O\Big( \Big(1+\frac{u}{T}\Big) (\log (qu(T+|t|)))^2 + \sum_{|\rho|\le 1} \frac{1}{|\rho|} \Big).\\
\end{eqnarray*}
\end{lemma}

\proof This unconditional result may be derived by following 
the method given  in Chapters 17 and 19 of Davenport \cite{DM80}.   We start with Perron's formula 
\beql{407b}
\frac{1}{2\pi i} \int_{1+1/\log u-i\infty}^{1+1/\log u+i\infty} -\frac{L^{\prime}}{L}(w+it, \chi) \frac{u^w}{w} 
dw = \sum_{n\le u} \Lambda(n) \chi(n)n^{-it} +O(\log u),
\eeq
Now for each $T\ge 2$ we may find $T_1$ and $T_2$ with $|T_1+T|\le 1$ 
and $|T_2-T|\le 1$ such that $|L^{\prime}/L(c+iT_j+it) | \ll (\log (q(T+|t|)))^2$ for 
all $-\frac 12 \le c\le 1+1/\log x$.  We truncate the integral in 
(\ref{407b}) to the line segment $[1+1/\log u +iT_1, 1+1/\log u + iT_2]$ 
and incur an error of $O(u(\log u)^2/T)$.  We now shift the line of
integration to  Re$(w)= -\frac{1}{2}$, using a rectangular contour.
In view of our choice for the heights $T_1$ and $T_2$, the horizontal sides 
contribute $O( u (\log (q(T+|t|))^2/T)$.  The vertical side of the box with Re$(w)=-\frac{1}{2}$ 
contributes $O( (\log qu(T+|t|))^2/\sqrt{u})$, upon using the functional equation 
to estimate $L^{\prime}/L$ on this line.  The net contribution of the 
error terms discussed so far is 
$$ 
\ll \Big(\frac{u}{T}+1\Big) (\log (qu(T+|t|)))^2. 
$$

It remains lastly to discuss the 
residues of the poles encountered while shifting our contour.  If $q=1$ and 
$\chi$ is the principal character, there is a pole at $w=1-it$ which leaves 
the residue $u^{1-it}/(1-it)$.   If $\rho$ is a zero of $L(s,\chi)$ with 
$0<\beta<1$ and $T_1\le \gamma -t \le T_2$ then there is a 
pole at $w=\rho-it$ in our contour shift.  The contribution of these poles is 
$$ 
-\sum_{{0<\beta <1}\atop {T_1< \gamma-t \le T_2}} \frac{u^{\rho-it}}{\rho-it} 
= - \sum_{{0< \beta<1} \atop {|\gamma-t|\le T}} \frac{u^{\rho-it}}{\rho-it} + 
O\Big( \frac{u}{T} \log (q(T+|t|))\Big),
$$ 
since the conditions $T_1 < \gamma -t < T_2$ and $|\gamma-t| \le T$ are 
different for at most $\ll \log (q(T+|t|))$ zeros.  
 Finally there is a pole at $w=0$ and, if $\chi(-1)=1$, $q>1$ and $-t \in [T_1,T_2]$ a 
 pole at $w=-it$.  The residues at these poles may be treated as in 
 Chapter 19 of Davenport \cite{DM80} and they contribute an amount $\ll \log (qu(T+|t|)) 
 + \sum_{|\rho|\le 1} 1/|\rho|$.  This sum over $|\rho|\le 1$ is to account 
 for the case where there is a Siegel zero very near $1$ (and hence a corresponding zero 
 very near $0$).  

Assembling these observations together, we obtain the Lemma.
\endproof
\smallskip

%
%
\begin{lemma}~\label{le42}  Assume the truth of  the GRH. 
If $\chi \imod q$ is a primitive Dirichlet character with $q>1$, then for $u \ge 1$ and all real $t$ 
we have 
\beql{421}
\sum_{n\le u} \Lambda(n) \chi(n)n^{-it} \ll \sqrt{u} (\log u) \log (qu(|t|+2))   .
\eeq
In the case of the principal character (and so $q=1$), we have for $u \ge 1$ and all real $t$,
\beql{422}
\sum_{n\le u} \Lambda(n) n^{-it} = \frac{ u^{1-it}}{1-it} + 
O( \sqrt{u} (\log u )\log( u(|t|+2)) ).
\eeq
\end{lemma}

\proof We apply Lemma~\ref{le41} choosing $T=u^2$.  We shall use GRH to bound 
the sums over zeros appearing there, and recall that there are $\ll (\log (q(2+|z|))$ 
zeros of $L(s, \chi)$  in $|\gamma-z|\le 1$.  Thus we obtain that 
\begin{eqnarray*}
\sum_{n\le u} \Lambda(n) \chi(n) n^{-it} 
&=& \delta(\chi) \frac{u^{1-it}}{1-it}  + 
O\Big( \sum_{|\gamma -t|\le T} \frac{\sqrt{u}}{1+|t-\gamma|} + (\log (qu (2+|t|)))^2 \Big)\\
&=&  \delta(\chi) \frac{u^{1-it}}{1-it}  + O\Big( \sqrt{u} (\log (qu(2+|t|))) \log u+ (\log (qu(2+|t|)))^2\Big).\\
\end{eqnarray*}   
If $\log (qu(2+|t|) \le \sqrt{u}$ then the second error term above 
may be absorbed into the first, and our Lemma follows.  If $\log (qu(2+|t|) \ge \sqrt{u}$ 
then the stated estimates are weaker than
the trivial bound $\sum_{n\le u} \Lambda(n)\chi(n) n^{-it} 
\ll u$, and so our Lemma holds in this case also.  

  \endproof
  
%
%
\proof[Proof of Proposition \ref{le43}]  From the definition of $L(s,\chi; y)$ we have 
that 
$$ 
|L(s,\chi;y)| = \exp\Big({\rm Re} (\log L(s, \chi;y))\Big) 
\ll  \exp\Big({\rm Re} \sum_{n \le y} \frac{\Lambda(n) \chi(n)n^{-it}}{n^{\sigma} \log n} \Big).
$$
If $(\log (q(2+|t|)))^2 \ge y$, then using the prime number theorem we have that 
\beql{518a}
\sum_{n\le y} \frac{\Lambda(n) \chi(n)n^{-it}}{n^{\sigma} \log n} \ll 
\sum_{n\le (\log (q(2+|t|)))^2} \frac{\Lambda(n)}{\sqrt{n} \log n} \ll \frac{\log (q(2+|t|))}{\log \log (q(2+|t|))}, 
\eeq  
and the bounds of the Lemma hold.  

Suppose now that $y\ge (\log (q(2+|t|)))^2$. We use the estimate \eqn{518a} above 
for the terms $n\le (\log (q(2+|t|)))^2$, and use partial summation and Lemma \ref{le42} 
for larger values of $n$.  Thus we find that 
$$ 
\sum_{n\le y} \frac{\Lambda(n) \chi(n)n^{-it}}{n^{\sigma} \log n} = 
O\Big(\frac{\log (q(2+|t|))}{\log \log (q(2+|t|))} \Big) + \int_{(\log q(2+|t|))^2}^{y} 
\frac{1}{z^{\sigma}\log z} d\Big(\sum_{n\le z} \Lambda(n)n^{-it} \chi(n)\Big).
$$ 
Suppose first that $q>1$.  Integrating by parts, and using (\ref{421}) we 
see that the integral above is 
\begin{eqnarray*}
&\ll& (\log (q(2+|t|)))^{2-2\sigma} + \int_{(\log (q(2+|t|)))^2}^{y} \sqrt{z}(\log (qz(2+|t|)))
\Big(\frac{\sigma}{z^{\sigma+1}} +\frac{1}{z^{\sigma+1}\log z}\Big) dz \\
&\ll& \frac{\sigma}{\sigma-\frac 12} (\log (q(2+|t|)))^{2-2\sigma}.\\
\end{eqnarray*}
If $\sigma \ge \frac 12+\epsilon$ then the above estimates readily imply (\ref{431}).  

The case when $q=1$ is similar, but we appeal to (\ref{422}) in place of 
(\ref{421}) above.  This leads to including an extra
main term in our sum above of size $y^{1-\sigma+it}/((1-it)\log y)$, 
and thus we obtain  (\ref{432}).  
   
\endproof

%
%
\section{ The weighted exponential sum $E(x,y;\alpha)$}\label{sec5}
\setcounter{equation}{0}

\noindent Our aim in this section is to understand the weighted sum $E(x, y; \alpha) 
=\sum_{n\in {\mathcal S}(y)} e(n\alpha) \Phi(n/x)$.  We shall use the 
decomposition into multiplicative characters developed in \S 3 
together with the GRH bounds for partial $L$-functions developed in \S 5. 
Here $\Phi$ is treated as fixed, and all constants in $O$-symbols 
depend on it. \smallskip

%
%
\begin{prop} \label{pr42}  Assume the truth of the GRH.  Let $\alpha$ be a 
real number in $[0,1]$ and write $\alpha =a/q + \gamma$ with $(a,q)=1$, 
$q\le \sqrt{x}$, and $|\gamma| \le 1/(q\sqrt{x})$.   Then 
\beql{401a}
E(x,y;\alpha) = M(x,y;q,\gamma) + O(x^{\frac{3}{4}+\epsilon}), 
\eeq
where the ``local main term" $M(x,y;q, \gamma)$ is defined by 
\beql{401b}
M(x,y; q, \gamma) = \sum_{n \in \sS(y)} 
\frac{\mu(\frac{q}{(q,n)})}{\phi(\frac{q}{(q,n)})}e(n \gamma) \Phi\Big(\frac{n}{x}\Big).
\eeq
\end{prop} 

\proof We begin by remarking that Dirichlet's theorem on diophantine approximation 
guarantees the existence of decompositions $\alpha=a/q+\gamma$ with $(a,q)=1$, 
$q\le \sqrt{x}$ and $|\gamma|\le 1/(q\sqrt{x})$.  Writing $n\in {\mathcal S}(y)$ as 
$dm$ where $d= (n,q)$
we see that 
\[ 
E(x,y;\alpha) = \sum_{{d|q}\atop{d \in \sS(y)}} 
\sum_{{m \in \sS(y)}\atop{(m, \frac{q}{d})=1}}
e\Big(\frac{am}{q/d}\Big) e(md\gamma) \Phi\Big(\frac{md}{x}\Big) .
\]
Using now Lemma~\ref{le31} we find that 
\beql{324}
E(x, y;  \alpha)=
\sum_{{d|q}\atop{d \in \sS(y)}} \frac{1}{\phi(q/d)} 
\sum_{\chi\!\! \!\!\imod{q/d}} \chi(a) \tau(\bar{\chi}) 
\sum_{m \in \sS(y)\atop{(m, \frac{q}{d})=1}} e(md \gamma) \chi(m)
\Phi\Big(\frac{md}{x}\Big).
\eeq 

Consider first the contribution of the principal 
character $\imod {q/d}$.  The Gauss sum for the principal character $\imod {q/d}$ 
equals $\mu(q/d)$, and hence the contribution of the principal characters 
to (\ref{324}) is 
\[
\sum_{{d|q}\atop{d \in \sS(y)}} \frac{\mu(q/d)}{\phi(q/d)}
 \sum_{m \in \sS(y)\atop{(m, \frac{q}{d})=1}} e(md \gamma) 
\Phi\Big(\frac{md}{x}\Big) = \sum_{n\in {\mathcal S}(y)} \frac{\mu(q/(q,n))}{\phi(q/(q,n))} 
e(n\gamma) \Phi\Big(\frac nx \Big) = M(x,y;q,\gamma).  
\]
This is the main term isolated in our Proposition, and we must show that 
the contribution of the non-principal characters to (\ref{324}) is $O(x^{\frac 34+\epsilon})$.   

We shall establish using Proposition \ref{le43} that if $\chi$ is not the principal character $\imod{q/d}$ then 
\beql{411} 
\frac {\sqrt q}{\sqrt d} \Big|\sum_{m \in \sS(y)} e(md \gamma) \chi(m) 
\Phi\Big(\frac{md}{x}\Big) \Big| \ll x^{\frac 34+\epsilon}.
\eeq 
Assuming this for the present, since $|\tau(\overline{\chi})| \le \sqrt{q/d}$ 
for all characters $\chi\imod{q/d}$ by Lemma \ref{le31a}, we see that the contribution of the non-principal 
characters to (\ref{324}) is bounded by 
\[
\ll \sum_{{d|q}\atop{d \in \sS(y)}} \frac{1}{\phi(q/d)} \sum_{{\chi \imod {q/d}} 
\atop {\chi \neq \chi_0}}   x^{\frac 34+\epsilon} \ll x^{\frac 34+\epsilon} d(q)
\ll x^{\frac 34+\epsilon}.  
\]

Thus to finish the proof of our Proposition, we need only establish (\ref{411}).  
 Using Lemma \ref{le32} we see that for any $c>0$ 
 \begin{eqnarray*} 
 \sum_{m \in {\mathcal S}(y)} e(md\gamma) \chi(m) \Phi\Big( 
 \frac{md}{x}\Big) &=& \sum_{m \in {\mathcal S}(y)} \frac{1}{2\pi i} \int_{c-i\infty}^{c+i\infty}  
 {\check \Phi}(s,\gamma x) \Big(\frac{x}{dm}\Big)^s ds\\
 &=& 
 \frac{1}{2\pi i} \int_{c-i\infty}^{c+i\infty} L(s,\chi; y) {\check \Phi}(s,\gamma x) \Big(\frac{x}{d}\Big)^s ds\\
 \end{eqnarray*} 
 where the interchange of the sum and integral is justified by the absolute convergence 
 of $L(s,\chi; y)$ for any Re$(s)>0$.  We now take $c=1/2+\epsilon$ and 
 invoke the GRH bound from Proposition \ref{le43}  which gives $L(s,\chi; y) \ll (q|s|)^{\epsilon}$.  
 Note that Proposition \ref{le43} applies to primitive characters $\chi$, but we may extend it 
 easily to imprimitive characters as follows.  Suppose $\chi$ is induced from a primitive character 
 ${\tilde \chi}\imod {\tilde{q}}$ then we have 
 $|L(s,\chi;y)| \le |L(s,{\tilde \chi};y)| \prod_{p|(q/{\tilde q})} (1+1/\sqrt{p}) \ll (q|s|)^{\epsilon}$ 
 upon using the bound of Proposition \ref{le43} for $L(s,{\tilde \chi};y)$.   
 It follows that 
 \[
\Big |  \sum_{m \in {\mathcal S}(y)} e(md\gamma) \chi(m) \Phi\Big(  \frac{md}{x}\Big)\Big |
  \ll \Big(\frac{x}{d}\Big)^{\frac 12+\epsilon} q^{\epsilon} 
 \int_{-\infty}^{\infty} |{\check \Phi}(\tfrac 12+\epsilon +it, \gamma x)| (1+|t|)^{\epsilon} dt.  
 \]
Using Lemma 3.5, we conclude that 
\[
\frac{\sqrt{q}}{\sqrt d} \Big| \sum_{m\in {\mathcal S}(y)} e(md\gamma)\chi(m) \Phi\Big(\frac{md}{x}\Big) \Big|
 \ll (\frac{1}{d})^{1+\epsilon}x^{\frac 12+\epsilon} q^{\frac 12+ \epsilon} (1+|\gamma| x)^{\frac 12+\epsilon} \ll x^{\frac 34+\epsilon}, 
\] 
since $q\le \sqrt{x}$ and $q|\gamma x| \le \sqrt{x}$.  This establishes 
(\ref{411}) and hence our Proposition.

\endproof

We now consider the ``local main terms" $M(x,y;q,\gamma)$, 
and start with a simple reduction.
\smallskip

%
%
\begin{lemma}~\label{le62} 
Given a positive integer $q$, write  $q=q_0 q_1$,
in which $q_0 \in \sS(y)$ and $q_1$ is divisible 
only by primes larger than $y$. Let
$M(x,y;q,\gamma)$ be as in Proposition \ref{pr42}. Then
\beql{621a}
M(x,y;q,\gamma) = \frac{\mu(q_1)}{\phi(q_1)} M(x,y;q_0,\gamma).
\eeq
\end{lemma}

\begin{proof}
This  is immediate from the definition \eqn{401b}.
\end{proof}

It remains to treat the case $q_0 \in \sS(y)$, and here we  use the  saddle 
point method of Hildebrand and Tenenbaum discussed in \S 4 to obtain 
an understanding of this main term.  In the following result  the
lower bound $y \ge (\log x)^{2+\delta}$ is imposed only as a necessary condition
for nontriviality of the  estimate. 
\smallskip

%
%
\begin{prop}\label{prop62}  Assume the truth of the GRH. 
Let $x$ and $y$ be large, and assume that 
$(\log x)^{2 +\delta} \le y\le \exp((\log x)^{\frac{1}{2}-\delta})$. 
Let $c=c(x,y)$ denote the 
 Hildebrand-Tenenbaum saddle point value given in section 4.
 Suppose $q_0 \in \sS(y)$ with $q_0< \sqrt{x}$,
 let $\gamma$ be real with $|\gamma| \le 1/(q_0\sqrt{x})$,  
and let $M(x,y;q_0,\gamma)$ be as in Proposition \ref{pr42}.   
Then we have:

(1) If $|\gamma| \ge x^{\delta -1}$ then, for any fixed $\epsilon>0$, 
$$
|M(x,y;q_0, \gamma)| \ll x^{\frac{3}{4} + \epsilon} q_0^{-\frac{3}{4}+\epsilon}.
$$ 

(2) If $|\gamma| \le x^{\delta-1}$ we have , for any fixed $\epsilon>0$, 
\begin{eqnarray*}
 M(x,y;q_0,\gamma)& =&   \frac{1}{q_0^c} \prod_{p|q_0} 
\Big( 1-\frac{p^c-1}{p-1}\Big) (c{\check \Phi}(c,\gamma x)) \Psi(x,y) + 
O_{\epsilon}(x^{\frac{3}{4}+\epsilon}q_0^{-\frac{3}{4}+\epsilon})\\
&&\hskip 1 in + O_{\epsilon}\Big(\frac{\Psi(x,y) q_0^{-c+\epsilon}}{(\log y)(1+|\gamma| x)^2}\Big).
\end{eqnarray*}
\end{prop}   
 
\proof 
Using Lemma \ref{le32} we see that for any $\sigma >0$ we have 
$$ 
M(x,y;q_0,\gamma) = \frac{1}{2\pi i}\int_{\sigma-i\infty}^{\sigma+i\infty} 
\sum_{n \in {\mathcal S}(y)} \frac{\mu(q_0/(q_0,n))}{\phi(q_0/(q_0,n))} \frac{1}{n^s} {\check \Phi}(s,\gamma x)x^s ds. 
$$ 
We now may write
$$ 
\sum_{n \in {\mathcal S}(y)} \frac{\mu(q_0/(q_0,n))}{\phi(q_0/(q_0,n))} \frac{1}{n^s}  
= \zeta(s;y) H(s; q_0),
$$
where $H(s; q_0)$ is  a Dirichlet series involving only integers with prime factors dividing $q_0$.
For each prime $p|q_0$ let $\nu_p(q_0)$ denote the exact power of $p$ dividing $q_0$, 
so that $\nu_p(q_0)\ge 1$.   Then 
\begin{eqnarray}\label{M0}
H(s;q_0) &= & \prod_{p|q_0}\Big(1-\frac{1}{p^s}\Big)\Big( \sum_{k=\nu_{p}(q_0)-1}^{\infty} \frac{\mu(p^{\nu_p(q_0)}/(p^{\nu_p(q_0)},p^k))}{\phi(p^{\nu_p(q_0)}/(p^{\nu_p(q_0)},p^k))} \frac 1{p^{ks}} \Big)
\nonumber \\
&= & \prod_{p|q_0}\Big(1-\frac 1{p^s}\Big) \Big( -\frac{1}{(p-1)} \frac{1}{p^{(\nu_p(q_0)-1)s}} + \frac{1}{p^{\nu_p(q_0)s}}\Big(1-\frac{1}{p^s}\Big)^{-1}\Big)
\nonumber  \\
&= & \frac{1}{q_0^s} \prod_{p|q_0} \Big(1- \frac{p^s-1}{p-1}\Big). 
\end{eqnarray}

We may now write our integral formula as
\beql{M1}
M(x,y;q_0,\gamma) =  \frac{1}{2\pi i} \int_{\sigma-i\infty}^{\sigma+i\infty} 
\zeta(s;y)H(s;q_0)  {\check \Phi}(s,\gamma x) x^sds.  
\eeq

We deform the integral above, replacing it by an integral over a piecewise linear contour 
consisting of (i) a line segment $c_1+it$ with $t$ going from $-y$ to $y$, (ii) a horizontal line segment going 
from $c_1+ iy$ to $1/2+\epsilon+iy$ and another going from $1/2+\epsilon-iy$ to $c_1-iy$, and 
(iii) a vertical line segment going from $1/2+\epsilon+iy$ to $1/2+\epsilon +i\infty$ and another 
going from $1/2+\epsilon-i\infty$ to $1/2+\epsilon -iy$.  
The shift of contour is permitted because the integrand is holomorphic and bounded 
in vertical strips $0< \sigma_1< \text{Re}(s)< \sigma_2$ and
is rapidly decreasing as $|\text{Im}(s)| \to \infty$ using the bound of  Lemma \ref{le33}.
The  proofs of  (1) and (2)  will choose different values of $c_1$.
In the calculations 
below it will be useful to keep in mind that for all $s$ with $1/2\le {\text{Re}}(s)\le 1$ 
we have 
$$
|H(s;q_0)| \le d(q_0) q_0^{-{\rm Re}(s)} \ll q_0^{-{\rm Re}(s)+\epsilon}.
$$

We consider first the vertical line segments given in case (iii) above,
which do not depend on the choice of $c_1$.  Using Proposition 
\ref{le43} (which assumes GRH),   
we see that the contribution of these segments to $M(x,y;q_0,\gamma)$ is 
\begin{eqnarray}\label{M2}
&\ll &  q_0^{-\frac 12+\epsilon} x^{\frac 12+\epsilon} \int_{-\infty}^{\infty} 
(1+|t|)^{\epsilon} |{\check \Phi}(\tfrac 12+\epsilon+it,\gamma x)| dt \nonumber \\
&\ll &q_0^{-\frac 12+\epsilon} x^{\frac 12+\epsilon} (1+|\gamma x|)^{\frac{1}{2}+\epsilon} 
\ll q_0^{-\frac 34+\epsilon} x^{\frac 34+\epsilon}, 
\end{eqnarray}
upon using Lemma \ref{le34} and that $|\gamma|x\le \sqrt{x}/q_0$
To handle the remaining integrals, we distinguish two cases depending 
on whether $|\gamma| \ge x^{\delta- 1}$ or not.  

(1) First we treat the case when $|\gamma|\ge x^{\delta-1}$. 
In this case we will choose  $c_1= 1+\epsilon$. 
Taking $k$ suitably large in Lemma \ref{le33} (depending on $\delta$) we find that 
${\check \Phi}(s,\gamma x) \ll x^{-1}$ for all $s$ on the portions of the 
contour given in (i) and (ii) above.   Consider the contribution
to the integral of the horizontal line 
segments in (ii).  Proposition \ref{le43} gives that 
\beql{677}
|\zeta(s;y)| \ll_{\epsilon}   \exp\Big(\frac{y^{1-\sigma}}{(1+y)\log y}\Big) |s|^{\epsilon}
 \ll_{\epsilon}    |s|^{\epsilon} \ll y^\epsilon, 
\eeq
and so this contribution is 
\[ 
\ll y^{\epsilon} x^{-1} \int_{1/2+\epsilon}^{c_1} x^{\sigma} q_0^{-\sigma+\epsilon} 
d\sigma \ll x^{\epsilon}.
\] 
Next consider the vertical line segment given in (i).  Here we bound 
$|\zeta(s;y)|$ by $\zeta(c_1,y) \ll_{\epsilon} 1$, so this segment 
 contributes 
\[
\ll x^{c_1} q_0^{-c+\epsilon} x^{-1} \zeta(c_1;y) y \ll_{\epsilon} x^{\epsilon}. 
\] 
Combining these estimates with \eqn{M2} we conclude that when $|\gamma|\ge x^{\delta -1}$ 
we have $M(x,y;q_0 ,\gamma) \ll x^{\frac 34+\epsilon} q_0^{-\frac 34+\epsilon}$, as claimed.

(2) Now we turn to the case when $|\gamma| \le x^{\delta-1}$.  
In this case we choose $c_1=c$ to be the Hildebrand-Tenenbaum saddle point value.
We use Lemma 
\ref{le33} with $k=2$ which gives that ${\check \Phi}(s,\gamma) \ll |s|^2/(1+|\gamma|x)^2$.  
Now consider  the contribution to the integral of the horizontal line segments described in (ii).  As
 in \eqn{677} above, 
Proposition \ref{le43} gives that  $|\zeta(s;y)| \ll |s|^{\epsilon}\ll y^{\epsilon}$, and 
so the contribution of these line 
segments to the integral giving $M(x,y;q_0,\gamma)$ is 
$$ 
\ll y^{\epsilon} \frac{y^2}{(1+|\gamma|x)^2}  \int_{1/2+\epsilon}^{c} x^{\sigma} q_0^{-\sigma+\epsilon} d\sigma 
\ll \frac{y^{2+\epsilon} }{(1+|\gamma|x)^2} q_0^{-c+\epsilon}  x^{c}.
$$
Note that, using (\ref{S802}),
$$
\log \zeta(c;y) \ge \sum_{p\le y} p^{-c} \ge \frac 1{2\log y} \sum_{p\le y} \frac{\log p}{p^c-1} 
= \frac{\log x}{2\log y},
$$ 
and so the contribution of the horizontal line segments is, using Theorems 
\ref{th91} and \ref{th92} 
$$
\ll \frac{ y^{2+\epsilon}}{(1+|\gamma|x)^2} q_0^{-c+\epsilon}  x^{c} 
\zeta(c;y)\exp\Big(-\frac {\log x}{2\log y}\Big) 
\ll q_0^{-c+\epsilon} \Psi(x,y) \frac{y^{2+\epsilon} \log x}{(1+|\gamma|x)^2} \exp\Big(-\frac{\log x}{2\log y}\Big).
$$ 
(Here we used the bound $\sqrt{2\pi \phi_2(c, y)} \ll \log x$ from Theorem \ref{th92}
and $y> (\log x)^{1+\epsilon}$.)
Since $y\le \exp((\log x)^{\frac 14})$, this yields the bound
\beql{M3}
\ll \frac{ q_0^{-c+\epsilon}   \Psi(x,y)}{(\log y)^3 (1+|\gamma|x)^2},
\eeq 
with plenty to spare.

Finally we consider the contribution of the vertical line segment given in (i).  We split this integral 
into the regions $|t|\le 1/\log y$ and $1/\log y \le |t| \le y$.  
We  first treat the saddle-point region $|t|\le 1/\log y$ lying near the real axis, which
 contributes to the   
main term of the formula in (2).  Certainly $c+it = c+O(|t|)$, and we may check easily that 
for $|t| \le 1$
$$ 
|H(c+it;q_0) - H(c;q_0) | \ll |t| q_0^{-c+\epsilon}.
$$ 
It is clear that 
$$ 
{\check \Phi}(c+it,\gamma x) -{\check \Phi}(c,\gamma x) =
\int_0^{\infty} \Phi(w) e(\gamma xw) (w^{c-1+it} -w^{c-1}) dw \ll |t|,
$$ 
and integrating by parts twice we also have 
$$ 
{\check \Phi}(c+it,\gamma x) - 
{\check \Phi}(c,\gamma x) =  \int_{0}^{\infty} \frac{d^2}{dw^2} \Big( \Phi(w)w^{c-1}(w^{it}-1) \Big)
 \frac{e(\gamma xw)}{(2\pi i x\gamma)^2} dw 
\ll \frac{|t|}{(|\gamma| x)^2}.
$$ 
We conclude that $|{\check \Phi}(c+it, \gamma x) - {\check \Phi}(c,\gamma x) | \ll |t|/(1+|\gamma|x)^2$.  Putting these observations together we see that for $|t|\le 1/\log y$, 
$$ 
|(c+it) {\check \Phi}(c+it, \gamma x) H(c+it;q_0) -c {\check \Phi}(c,\gamma x) H(c;q_0)| \ll
\frac{ q_0^{-c+\epsilon} }{(\log y) (1+ |\gamma | x)^2}. 
$$ 
Hence the contribution of the region $|t|\le 1/\log y$ to $M(x,y;q_0,y)$ is 
\begin{eqnarray}\label{M4}
&&\frac{1}{2\pi i} \int_{c-i/\log y}^{c+i/\log y} 
\zeta(s;y) \frac{x^s}{s} \Big( cH(c;q_0) {\check \Phi}(c,\gamma x) + 
O_{\epsilon} \Big( \frac{q_0^{-c+\epsilon} }{(\log y)(1+|\gamma|x)^2}\Big)\Big) ds 
\nonumber  \\ 
&=& ~c H(c;q_0) {\check \Phi}(c,\gamma x) \Psi(x,y) + 
O_{\epsilon}\Big( \frac{q_0^{-c+\epsilon}}{(\log y)(1+|\gamma| x)^2} \Psi(x,y)\Big), 
\end{eqnarray}
upon using Lemma \ref{le92} to produce the main term and 
the absolute value integral in Lemma \ref{le93} to bound the error term.

Next consider the remaining region $1/\log y \le |t| \le y$ in segment (i).  Bounding 
the absolute value of the integrand, and using Lemma \ref{le91} and that $(1-c)\ll (\log \log x)/\log y$ 
(see \eqref{908}), 
this contribution is   
\begin{align*}
&\ll x^{c} q_0^{-c+\epsilon} \frac{ y^3}{(1+|\gamma| x)^2}  \max_{1/\log y \le | t| \le y} |\zeta(c+it;y)|
 \\
 &\ll q_0^{-c+\epsilon} 
\frac{y^3}{(1+|\gamma|x)^2} x^{c}\zeta(c;y) \exp\Big( - C \frac{\log x}{(\log y)(\log \log x)^2}\Big),   
\end{align*} 
for some positive constant $C$.  Appealing now to Theorems \ref{th91} and \ref{th92}, and using 
$y\le \exp((\log x)^{\frac{1}{2}-\delta})$, we deduce 
that the above is bounded by
\[
\ll q_0^{-c+\epsilon}  \Psi(x,y) \frac{y^3 \log x }{(1+|\gamma|x)^2}\exp\Big(-C \frac{\log x}{(\log y)(\log \log x)^2}\Big) 
\le \frac{q_0^{-c+\epsilon}\Psi(x,y)}{ (\log y)^{3}(1+|\gamma|x)^2}.
\]
 Combining this bound with 
(\ref{M0}), (\ref{M1}), (\ref{M2}), (\ref{M3}), (\ref{M4}), 
we obtain the estimate of the Proposition in the 
case when $|\gamma|\le x^{\delta-1}$.

\endproof
\smallskip

%
%
\proof[Proof of Theorem \ref{th2.3}] 
We use Proposition \ref{pr42}
and Lemma \ref{le43}, writing $q=q_0q_1$  to obtain
\beql{691}
E(x, y; \alpha) = \frac{\mu(q_1)}{\phi(q_1)} M(x, y; q_0, \gamma) + 
O\Big(x^{\frac{3}{4} + \epsilon}\Big).
\eeq
Now Proposition \ref{prop62}~(1) applied to $M(x,y, q_0, \gamma)$ gives the bound of part (1).
Next  we note that
 the Hildebrand-Tenenbaum saddle point $c=c(x,y)$ satisfies 
$c= 1-1/\kappa +O(\log \log y/\log y)$ (see (\ref{819a})),
 and so for $c_0=1-\frac{1}{\kappa}$ we have
\begin{eqnarray*}
\frac{1}{q_0^c}\prod_{p|q_0} \Big( 1- \frac{p^c-1}{p-1} \Big) (c{\check \Phi}(c,\gamma x) ) 
 &= &\frac{1}{q_0^{c_0}}\prod_{p|q_0} \Big(1- \frac{p^{c_0}-1}{p-1}\Big) (c_0 {\check\Phi}(c_0,\gamma x)) \\
 &&+ 
 O_{\epsilon}\Big( \frac{q_0^{-c_0+\epsilon}}{(1+|\gamma| x)^2} \frac{\log \log y}{\log y}\Big).
 \end{eqnarray*}
  Now part (2) follows upon using  this  formula in the expression for $M(x, y, q_0, \gamma)$ in
  Proposition \ref{prop62}~(2),   substituting the result into \eqn{691},
  noting that $\phi(q_1) \gg (q_1)^{-1+\epsilon}$.
\endproof

%
%
%

\section{ Counting Weighted Smooth Solutions: Proof of Theorem \ref{th21} }
 \setcounter{equation}{0}

\noindent We initially suppose that $(\log x)^{2+\delta} \le y \le \exp ( (\log x)^{\frac 12-\delta})$,
and we shall raise the lower bound on $y$ as the proof progresses.
We employ the Hardy-Littlewood circle method to evaluate 
$$ 
N(x,y;\Phi) = \int_0^1 E(x,y;\alpha)^2 E(x,y;-\alpha) d\alpha. 
$$ 
Let a fixed small number $\delta>0$  be given, which we use
as a parameter in defining major and minor arcs.
Given a rational number $a/q$ with $(a,q)=1$ and $q\le x^{\frac{1}{4}}$, 
we define the major arc centered at $a/q$ to be the set of all points 
$\alpha \in [0,1]$ with $|\alpha-a/q| \le x^{\delta-1}$.
 Note that any $\alpha \in [0,1]$ lies 
on at most one major arc.   We will find it convenient to group the major arcs $[0,x^{\delta -1}]$ 
and $[1-x^{\delta-1},1]$ together, and on ${\TT}={\RR}/{\ZZ}$ we may identify them with 
$[-x^{\delta-1},x^{\delta-1}]$.  The union of the major arcs is denoted ${\fM}$ and the 
minor arcs ${\fm}$ are defined to be the complement of 
the major arcs $[0,1]\backslash {\fM}$.

Suppose $\alpha$ lies on a minor arc.  By Dirichlet's theorem on Diophantine 
approximation we may write $\alpha=a/q+\gamma$ where $q\le \sqrt{x}$, $(a,q)=1$ 
and $|\gamma|\le 1/(q\sqrt{x})$.  Since $\alpha \in {\fm}$ we must have 
that either $q>x^{\frac 14}$, or that $|\gamma| \ge x^{\delta-1}$.  If the latter case holds
then, using Propositions \ref{pr42} and  \ref{prop62}, we find that 
$E(x,y;\alpha) \ll x^{\frac{3}{4}+\epsilon}$.   In the former case, Proposition \ref{prop62} 
with \eqn{621a} gives that 
\[
M(x,y;q,\gamma) \ll \Psi(x,y)q_0^{-c+\epsilon}q_1^{-1+\epsilon}  \ll x^{\frac 34+\epsilon}.
\]
Then by Proposition \ref{pr42} we have $E(x,y;\alpha) \ll x^{\frac 34+\epsilon}$.  
Thus $E(x,y;\alpha)\ll x^{\frac 34+\epsilon}$ when $\alpha$ lies on a minor arc.  Therefore 
\[ 
\int_{{\fm} } E(x,y;\alpha)^2 E(x,y;-\alpha) d\alpha \ll x^{\frac 34+\epsilon} 
\int_{\fm} |E(x,y;\alpha)|^2 d\alpha \ll x^{\frac 34+\epsilon} \int_0^1 |E(x,y;\alpha)|^2 
d\alpha,
\]
By Parseval, we have 
\[
\int_{0}^1 |E(x, y; \alpha)|^2 d \alpha= \sum_{n \in \sS(y)} \Big|\Phi\Big(\frac{n}{x}\Big)\Big|^2 
\ll E(x,y;0) 
\ll \Psi(x, y), 
\]
where the last inequality follows from Theorem \ref{th2.3}, or alternatively  from an 
application of Theorem \ref{th93}.
From this we obtain  the minor arc bound
\beql{7.1}
\int_{\fm} E(x,y;\alpha)^2 E(x,y;-\alpha) d\alpha 
\ll x^{\frac{3}{4}+\epsilon} \Psi(x,y). 
\eeq

It remains now to evaluate the major arc contribution.  If $z=z_1+O(z_2)$ then it 
follows that $|z|^2z = |z_1|^2 z_1 + O(|z_2| |z|^2)$.   Therefore if $\alpha$ lies 
on the major arc centered at $a/q$, Proposition \ref{pr42} gives that, with $\alpha =a/q+\gamma$ as 
before,  
\begin{eqnarray*}
E(x,y;\alpha)^2 E(x,y;-\alpha) &=& |E(x,y;\alpha)|^2 E(x,y;\alpha) 
\\
&=& |M(x,y;q,\gamma)|^2M(x,y;q,\gamma) 
+O_{\epsilon}(x^{3/4+\epsilon} |E(x,y;\alpha)|^2).
\\
\end{eqnarray*}
Thus the major arc contribution is 
$$ 
\sum_{q\le x^{\frac{1}{4}}} \sum_{{a=0}\atop {(a,q)=1}}^{q-1} \int_{-x^{\delta-1}}^{x^{\delta-1}}  
|M(x,y;q,\gamma)|^2 M(x,y;q,\gamma) d\gamma + O_{\epsilon}\Big(x^{\frac 34+\epsilon} \int_0^1 |E(x,y;\alpha)|^2 
d\alpha\Big),
$$ 
which we may simplify to 
\beql{7.2} 
\sum_{q\le x^{\frac{1}{4}}} \phi(q) \int_{-x^{\delta-1}}^{x^{\delta -1}} |M(x,y;q,\gamma)|^2 M(x,y;q,\gamma) 
d\gamma + O_{\epsilon}(x^{\frac{3}{4}+\epsilon} \Psi(x,y)). 
\eeq

Using the decomposition $q=q_0q_1$ of Lemma \ref{le62}
and the estimate of Proposition \ref{prop62} (2),
we find that  $|M(x,y;q,\gamma)|^2 M(x,y;q,\gamma)$ equals  
\begin{eqnarray*}
&& \frac{\mu(q_1)}{\phi(q_1)^3} 
\frac{1}{q_0^{3c}} \prod_{p|q_0} \Big(1- \frac{p^c-1}{p-1}\Big)^{3} c^3 |{\check \Phi}(c, \gamma x)|^2 
{\check \Phi}(c, \gamma x) \Psi(x,y)^3\\
&&\hskip 0.2 in + 
O_{\epsilon}\Big(\frac{1}{\phi(q_1)^3}  \sum_{j=1}^{3} \Big( x^{\frac 34+\epsilon} q_0^{-\frac 34+\epsilon} 
+ \frac{\Psi(x,y) q_0^{-c+\epsilon}}{(\log y)(1+|\gamma|x)^2}\Big)^{j} 
\Big(\frac{1}{q_0^c }\Psi(x,y) {\check \Phi}(c,\gamma x)\Big)^{3-j} \Big). \\
\end{eqnarray*}
Since $\sum_{j=1}^{3} A^j B^{3-j} \ll A^3 +A B^2$, 
and since $|{\check \Phi}(c, \gamma x)| = O(1)$ (using Lemma \ref{le33})
we may simplify the error term above to
\[
 \ll   \frac{1}{\phi(q_1)^3} \Big( x^{\frac{9}{4} + 3 \epsilon} q_0^{-\frac{9}{4}+ 3\epsilon}+
x^{\frac 34+\epsilon} q_0^{-2c-\frac 34+\epsilon} \Psi(x,y)^2 
 + \frac{\Psi(x,y)^3 q_0^{-3c+\epsilon}}{(\log y) 
 (1+|\gamma|x)^2}\Big).
\]
 To upper bound the contribution of this error term to (\ref{7.2}), we note that
 \[
 \int_{-x^{\delta -1}}^{x^{\delta -1}}\frac{d\gamma}{(1+ |\gamma x|)^2} \ll \frac{1}{x},
 \]
 and then we obtain
 \beql{7.3b}
 \ll \sum_{q\le x^{\frac{1}{4}}} \frac{\phi(q)}{\phi(q_1)^3} \Big(
  x^{\frac{5}{4}+\delta+\epsilon} q_0^{-\frac{9}{4}+\epsilon}  + 
x^{-\frac{1}{4}+\delta+\epsilon} q_0^{-2c-\frac{3}{4}+\epsilon} \Psi(x,y)^2 + \frac{\Psi(x,y)^3 q_0^{-3c+\epsilon}}{x\log y} \Big). 
 \eeq

We now raise the lower bound to $y\ge (\log x)^{4+8\delta}$.
We then have $\Psi(x,y) \ge x^{\frac{3}{4}+ \frac{3}{2}\delta}$, and
in addition the Hildebrand-Tenenbaum
saddle point $c> \frac{3}{4}$ by \eqn{819a}.
We deduce for $y$ in this range the 
error term contribution (\ref{7.3b}) above is 
\[
\ll \frac{\Psi(x,y)^3}{x\log y}. 
\]
We conclude that  for $y \ge (\log x)^{4 + 8\delta}$  the major arcs contribution is
\begin{eqnarray*}
&&\Psi(x,y)^3 \sum_{q\le x^{\frac 14}} 
\frac{\mu(q_1)}{\phi(q_1)^2} \frac{\phi(q_0)}{q_0^{3c}} \prod_{p|q_0} 
\Big(1-\frac{p^c-1}{p-1}\Big)^3 \int_{-x^{\delta-1}}^{x^{\delta-1}} 
c^3 |{\check \Phi}(c,\gamma x)|^2 {\check \Phi}(c,\gamma x) d\gamma 
 \\
&&\hskip 1 in + O_{\epsilon}\Big( x^{\frac{3}{4}+\epsilon} \Psi(x,y) + \frac{\Psi(x,y)^3}{x\log y}\Big).
\\
\end{eqnarray*}

Using Lemma \ref{le33} with $k=2$, and the Plancherel formula, we obtain that 
\begin{eqnarray*}
 \int_{-x^{\delta-1}}^{x^{\delta-1}} 
c^3 |{\check \Phi}(c,\gamma x)|^2& {\check \Phi}&(c,\gamma x) d\gamma 
= \frac{c^3}{x} \int_{-x^{\delta}}^{x^{\delta}} |{\check \Phi}(c,\xi)|^2 {\check \Phi}(c,\xi) 
d\xi 
\\
&=& \frac{c^3}{x} \Big( \int_{-\infty}^{\infty} |{\check \Phi}(c,\xi)|^2 {\check \Phi}(c,\xi) 
d\xi + O\Big( \int_{|\xi|> x^{\delta}} \frac{1}{1+\xi^2} d\xi \Big) 
\Big)\\ 
&=& \frac{c^3}{x} \Big( \int_0^{\infty} \int_{0}^{\infty} \Phi(t_1) \Phi(t_2) \Phi(t_1+t_2) 
(t_1 t_2 (t_1+t_2))^{c-1} dt_1 dt_2 + O(x^{-\delta})\Big) 
\\
&=& \frac{\fS_{\infty}(c,\Phi)}{x} + O(x^{-1-\delta}).
\\
\end{eqnarray*}

For $y \ge (\log x)^{4+8\delta}$, using $c \ge 3/4$ we see that 
\begin{eqnarray*} 
\sum_{q\le x^{\frac 14}} \frac{\mu(q_1)}{\phi(q_1)^2} \frac{\phi(q_0)}{q_0^{3c}}&&
\prod_{p|q_0} \Big(1-\frac{p^c-1}{p-1}\Big)^3 \\
&=&\sum_{q=1}^{\infty}  \frac{\mu(q_1)}{\phi(q_1)^2} \frac{\phi(q_0)}{q_0^{3c}} 
\prod_{p|q_0} \Big(1-\frac{p^c-1}{p-1}\Big)^3  + O\Big( \sum_{q>x^{\frac{1}{4}}} 
\frac{1}{q^{3c-1-\epsilon}}\Big) 
\\
&=&\Big( 
\sum_{q_0 \in \sS(y)} \frac{\phi(q_0)}{q_0^{3c}} 
\prod_{p|q_0} \Big(1-\frac{p^c-1}{p-1}\Big)^3 \Big) 
\Big(\sum_{{q_1}\atop{p|q_1 \Rightarrow p>y}} \frac{\mu(q_1)}{\phi(q_1)^2}\Big)
+O(x^{-\frac 1{16}}) 
\\
&=& \prod_{p\le y}\Big( 1+ \frac{p-1}{p(p^{3c-1}-1)}\Big(\frac{p-p^c}{p-1}\Big)^3\Big)
 \prod_{p>y}\Big( 1- \frac{1}{(p-1)^2}\Big) + O(x^{-\frac 1{16}}) 
\\
&=& \fS_f(c,y) + O(x^{-\frac 1{16}}). 
\end{eqnarray*}
Putting these remarks together, we conclude that for $y \ge (\log x)^{4+8\delta}$
the major arcs contribution is 
$$ 
\fS_\infty(c,\Phi)\fS_f(c,y)  \frac{\Psi(x,y)^3}{x} + 
O_{\epsilon}\Big( x^{\frac 34+\epsilon} \Psi(x,y) + \frac{\Psi(x,y)^3}{x\log y}\Big). 
$$ 
We  combine  this result with the minor arcs estimate  (\ref{7.1}) to conclude that 
\beql{7.3}
N(x,y;\Phi) =  \fS_\infty(c,\Phi) \fS_f(c,y) \frac{\Psi(x,y)^3}{x} + O_{\epsilon}\Big( x^{\frac 34+\epsilon} \Psi(x,y) 
+ \frac{\Psi(x,y)^3}{x\log y} \Big). 
\eeq

To obtain an asymptotic formula,  we now impose the lower bound
  $y\ge (\log x)^{8+ \delta}$ . Thus $\kappa \ge 8+\delta$,
so that $\Psi(x,y) =x^{1-1/\kappa+o(1)} >x^{\frac{7}{8}+\epsilon}$.
Now by (\ref{819a}) we know that
$c=1-1/\kappa +O(\log \log y/\log y)$.  Both $\fS_f(c,y)$ and $\fS_\infty(c,\Phi)$ are 
of constant size, and moreover we have
$$ 
\fS_f(c,y) = \fS\Big(1-\frac 1{\kappa},y\Big) + O\Big( \frac{\log \log y}{\log y}\Big),
$$ 
and 
$$
\fS_\infty(c,\Phi) = \fS_\infty\Big(1-\frac 1\kappa,\Phi\Big) + O\Big(\frac{\log \log y}{\log y}\Big). 
$$ 
We use these observations in (\ref{7.3}), and note also that 
the lower bound on $\Psi(x,y)$ above
implies that the error term $x^{\frac{3}{4}+\epsilon} \Psi(x,y)$ is subordinate 
to the error term $\Psi(x,y)^3/(x\log y)$, so that (\ref{7.3})
is an asymptotic formula. This proves Theorem \ref{th21}.

%
%
%
\section{Counting Weighted Primitive Smooth Solutions: Proof of Theorem \ref{th121}}
 \setcounter{equation}{0}
 
 \noindent 
 We suppose $(\log x)^{8+\delta} \le y \le \exp\Big( (\log x)^{\frac 12-\delta}\Big)$.
  Let $z=\frac{1}{2}\log y$, and put $P_z= \prod_{p\le z} p$.  By the 
 prime number theorem we know that $P_z =e^{z+o(z)} \le y$.  We assert that
\beql{1211}
\Big| N^{\ast}(x,y; \Phi) -
\sum_{d| P_z}\mu(d) N\Big(\frac{x}{d},y; \Phi\Big) \Big| \le 
\sum_{z< p \le y} N\Big(\frac{x}{p}, y; |\Phi| \Big).
\eeq
To establish (\ref{1211}), it suffices to observe that its left hand side 
counts weighted solutions  to $X+Y=Z$ with $XYZ \in \sS(y)$ and such that 
the gcd $(X, Y, Z)$ is an integer greater than $1$ and divisible 
only by primes larger than $z$.  
The proof will derive the desired
  asymptotic formula for the inclusion-exclusion sum
  $\sum_{d| P_z}\mu(d) N\Big(x/d,y; \Phi\Big)$ on the left side of \eqn{1211} and
   will complete the argument by showing  that the
  right side of \eqn{1211} is small compared to this asymptotic estimate.
  
 To handle the terms arising in (\ref{1211}) we first consider 
 $N(x/k,y;\Phi)$ and $N(x/k,y;|\Phi|)$ where $1\le k\le y$.  In our 
 range for $x$ and $y$, the exponent $\kappa=\kappa(x,y) := \log y/\log \log x$  satisfies
 $$
 \Big|\frac 1{\kappa(x,y)}-\frac{1}{\kappa(x/k,y)}\Big|  
 \le \frac{\log \log x - \log \log (x/y)}{\log y} \ll \frac{1}{\log x}, 
 $$ 
  and therefore 
 \begin{eqnarray*}
 && \fS_{\infty}\Big( 1-\frac{\log \log (x/k}{\log y},\Phi\Big)
 \fS_{f}\Big( 1-\frac{\log \log (x/k)}{\log y}, y\Big)  \\
&&\hskip 1 in  = \fS_{\infty} \Big(1-\frac{\log \log x}{\log y}, \Phi\Big) \fS_{f}\Big( 1-\frac{\log \log x}{\log y},y\Big) 
 + O\Big( \frac{1}{\log x}\Big).
 \\
 \end{eqnarray*}
 Furthermore, by Theorem \ref{th93}  we have that 
 \[ 
 \Psi(\frac{x}{k},y) = k^{-c(x/k,y)} \Psi(x,y) \Big( 1+ O\Big(\frac{\log y}{\log x}\Big) \Big).,
 \] 
 where $c(x/k, y)$ is the Hildebrand-Tenenbaum saddle point.
 Using Theorem \ref{th21} we  conclude that 
\begin{align}
\label{8.2}
N\Big(\frac{x}{k},y; \Phi\Big) 
&=  
\fS_{\infty}\Big(1-\frac{\log \log x}{\log y},\Phi\Big) \fS_{f}\Big(1-\frac{\log \log x}{\log y},y\Big)
\frac{\Psi(x,y)^3}{x}k^{1-3c(x/k,y)}\nonumber \\
&\hskip 1 in
+ O\Big( k^{1-3c(x/k,y)} \frac{\Psi(x,y)^3\log \log y}{x\log y} \Big). 
\end{align}
Similarly, we obtain the upper bound 
\beql{8.3}
N(\frac{x}{k}, y; |\Phi|) \ll k^{1-3c(x/k,y)} \frac{\Psi(x,y)^3}{x}. 
\eeq

We first bound the right hand side
of (\ref{1211}). We find using 
  (\ref{8.3})  that it is bounded by 
$$ 
\ll \frac{\Psi(x,y)^3}{x} \sum_{z< p \le y} p^{1-3c(x/p,y)},
$$ 
Since by (\ref{819a}) we have
\[
c(x/p,y)=1-1/\kappa(x/p,y) + O(\log \log y/\log y) = 1-1/\kappa +O(\log \log y/\log y) > 3/4,
\]
the above is bounded by
$$ 
\ll \frac{\Psi(x,y)^3}{x} \sum_{z<p\le y} p^{-5/4} \ll \frac{\Psi(x,y)^3}{xz^{\frac{1}{4}}}.
$$

Now, using (\ref{8.2}), we treat the sum on the left side  in \eqn{1211}, and find that 
\begin{align*} 
\sum_{d|P_z} \mu(d) N\Big( \frac{x}{d},y;\Phi\Big) 
&=  \fS_\infty\Big(1-\frac 1{\kappa},\Phi\Big) \fS_f\Big(1-\frac{1}{\kappa},y\Big)
\frac{\Psi(x,y)^3}{x} \Big( \sum_{d|P_z} \mu(d) d^{1-3c(x/d,y)} \Big)
\\
&\hskip 1in + O\Big( \frac{\Psi(x,y)^3 \log \log y}{x\log y}\Big( \sum_{d|P_z} d^{1-3c(x/d,y)}\Big)\Big).
\end{align*} 
Since $c(x/d,y) >3/4$, as noted above, the remainder term here is $O(\Psi(x,y)^3\log \log y/(x\log y))$. 
Next we  treat the sum appearing in this last estimate, and again we use that 
$c(x/d,y) = 1-1/\kappa+O(\log \log y/\log y)$. We obtain 
\begin{align*}
\sum_{d|P_z} \mu(d) d^{1-3c(x/d,y)} &= 
\sum_{d\le z} \mu(d) d^{3/\kappa-2}\Big( 1+O\Big(\frac{\log d \log \log y}{\log y}\Big)\Big) 
+ O\Big( \sum_{d>z} d^{3/\kappa -2 +o(1)}\Big) \\
&=\prod_{p\le z} \Big( 1- \frac{1}{p^{2-3/\kappa}}\Big) + O\Big( \frac{\log \log y}{\log y}\Big) 
+ O(z^{-1+3/\kappa+o(1)})\\
&=\prod_{p\le y}\Big( 1- \frac{1}{p^{2-3/\kappa}}\Big) +  O((\log y)^{-\frac{1}{2}}).
\end{align*}

We now define
\begin{eqnarray}\label{901}
\fS_f^{\ast}(c,y) 
& := & \fS_f(c, y) \prod_{p\le y}\Big(1- \frac{1}{p^{3c-1}}\Big)
\nonumber \\
&= & \prod_{p\le y} \Big( 1 + \frac{1}{p^{3c-1}}\Big( \frac{p-1}{p}\Big(\frac{p-p^c}{p-1}\Big)^3 -1\Big)
\prod_{p>y} \Big(1- \frac{1}{(p-1)^2}\Big).
\end{eqnarray}
Using this definition, substition of  the above estimates in (\ref{1211}) gives
\[
N^*(x,y;\Phi) =  \fS_\infty \Big(1-\frac{1}{\kappa},\Phi\Big) \fS_f^{\ast}\Big( 1- \frac 1{\kappa},y\Big)
\frac{\Psi(x,y)^3}{x} + O\Big( \frac{\Psi(x,y)^3}{x(\log y)^{\frac{1}{4}}}\Big). 
\]
This proves Theorem \ref{th121}.  \smallskip

\paragraph{\bf Acknowledgements.}\label{ackref}
The authors thank  Kalman Gy\'{o}ry, Peter Hegarty
and Michel  Waldschmidt  for helpful remarks. The authors thank the reviewer
for useful improvements.
 Some of this work was done when
the first author visited Stanford University, and he thanks the Mathematical Research
Center at Stanford University for support. Some final revisions were done while
the first author visited MSRI, whom he thanks for support.

%

\end{document}